\documentclass[reqno]{amsart}
\usepackage{caption}
\usepackage{subcaption}

\usepackage{mathabx,cite}
\usepackage[utf8]{inputenc} 
\usepackage[english]{babel}
\usepackage{amstext}
\usepackage{euscript}
\usepackage{bbm}
\usepackage{mathrsfs}
\usepackage{enumerate}
\usepackage{graphicx}
\usepackage{amscd}
\usepackage{verbatim}
\usepackage{amssymb}
\usepackage{amsthm}
\usepackage{amsmath}
\usepackage{amstext}
\usepackage{amsfonts}
\usepackage{latexsym}
\usepackage{mathtools} 
\usepackage{enumitem} 
\usepackage{accents} 
\usepackage[foot]{amsaddr} %package that puts the authors addresses on the title page as a footnote (when using AMSART document class only).

 \usepackage[left=1.2in, right=1.2in, top=1.5in, bottom=1.5in]{geometry}
\usepackage[usenames,dvipsnames]{xcolor} 
\definecolor{dkblue}{RGB}{1,31,91} % This is a dark Blue     
\usepackage[colorlinks=true, pdfstartview=FitV, linkcolor=dkblue, citecolor=dkblue, urlcolor=dkblue]{hyperref}

\usepackage{todonotes}
\newcommand{\CC}{\mathbb C}
\newcommand{\RR}{\mathbb R}

\newcommand{\NN}{\mathbb N}
\newcommand{\ZZ}{\mathbb Z}

\newcommand{\TT}{\mathbb T}

\newcommand{\sgn}{\operatorname{sgn}}

\normalsize
\normalsize
\theoremstyle{definition}
\newtheorem{theorem}{Theorem}

\newtheorem{lemma}[theorem]{Lemma}

\newtheorem{remark}[theorem]{Remark}

\numberwithin{equation}{section}
\numberwithin{theorem}{section}
\numberwithin{definition}{section}

\begin{document}

\subjclass[2010]{}

\title[A fluid-solid interaction problem in porous media]{A fluid-solid interaction problem in porous media}

\author[D. Alonso-Or\'an]{Diego Alonso-Or\'an}
\address{Departamento de An\'{a}lisis Matem\'{a}tico and Instituto de Matem\'aticas y Aplicaciones (IMAULL), Universidad de La Laguna, C/Astrof\'{i}sico Francisco S\'{a}nchez s/n, 38271, La Laguna, Spain. \href{mailto:dalonsoo@ull.edu.es}dalonsoo@ull.edu.es}

\author[R. Granero-Belinch\'on]{Rafael Granero-Belinch\'on}
\address{Departamento  de  Matem\'aticas,  Estad\'istica  y  Computaci\'on,  Universidad  de Cantabria.  Avda.  Los  Castros  s/n,  Santander,  Spain. \href{mailto:rafael.granero@unican.es}rafael.granero@unican.es}

\begin{abstract}
In this work, we derive asymptotic interface models for an elastic Muskat free boundary problem describing Darcy flow beneath an elastic membrane. In a weakly nonlinear regime of small interface steepness, we obtain nonlocal evolution equations that capture the free-boundary dynamics up to quadratic order. In the long-wave thin-film regime, we rewrite the kinematic condition in flux form, flatten the moving domain, and derive a lubrication-type equation. Moreover, we establish well-posedness for these models in suitable Wiener spaces.
\end{abstract}

\thispagestyle{empty}

\maketitle
\tableofcontents

\section{Introduction}\label{sec:intro}

Flow in porous media plays a fundamental role in a wide range of physical and industrial applications,
including groundwater dynamics, oil recovery, geothermal reservoirs, and filtration processes.
At the macroscopic level, slow viscous flows through a porous medium are commonly described by
Darcy’s law, which relates the fluid velocity to the pressure gradient and external forces.
When the flow domain involves a moving interface, Darcy’s law gives rise to free boundary problems,
among which the Muskat problem occupies a central position. \medskip

The Muskat problem models the evolution of an interface separating regions with different physical
properties inside a porous medium.
It originates in the study of immiscible fluids, but several reduced configurations are of independent
interest.
In particular, the one-phase Muskat problem describes the motion of a single incompressible,
viscous fluid in contact with a passive or empty region.
This setting arises naturally in hydrology, for instance in the description of thin water films above dry soil layers,
or in capillary zones where surface tension effects become relevant.
From a mathematical perspective, the Muskat problem is closely related to the Hele--Shaw flow,
and shares many structural features with free boundary problems in fluid mechanics. \medskip
 
In this work we consider a one-phase Muskat-type model in which the free interface is endowed with elastic effects. The fluid motion in the porous medium is governed by Darcy’s law, while the interface dynamics is driven by gravity and capillarity
and is further influenced by an elastic restoring force modeling, for instance, the interaction between the fluid and a thin elastic
sheet or membrane attached to the boundary. In this sense, the resulting elastic Muskat system may be viewed as a Darcy analogue
of capillary--gravity water waves in which surface tension is replaced by a higher-order (Willmore-type) bending energy, possibly
supplemented by a dissipative correction acting along the interface. \medskip

To make this setting precise, we work in a horizontally periodic geometry. Let $x=(x_1,x_2)\in(-L\pi,L\pi)^2$ denote the horizontal variables and let $z$ be the vertical coordinate.
At time $t$, the fluid occupies the region
\[
\Omega(t)=\big\{(x,z)\in\mathbb{R}^3:\,-d<z<h(x,t)\big\},
\]
with free surface
\[
\Gamma(t)=\big\{(x,h(x,t)):\,x\in(-L\pi,L\pi)^2\big\},
\]
and flat impermeable bottom
\[
\Gamma_{\mathrm{bot}}=\big\{(x,-d):\,x\in(-L\pi,L\pi)^2\big\}.
\]
We impose periodic boundary conditions in the horizontal variables.
The unknowns are the Darcy velocity field $u=u(x,z,t)$, the pressure $p=p(x,z,t)$,
and the interface height $h=h(x,t)$.

In Eulerian variables, the elastic Muskat system consists of Darcy’s law with gravity,
the incompressibility constraint, a dynamic boundary condition coupling elastic and dissipative effects,
the kinematic condition at the interface, and an impermeability condition at the bottom:
\begin{align}
	\frac{\mu}{\kappa}u+\nabla_{x,z}p&=-\chi\rho G\,e_3,
	&&\text{in }\Omega(t)\times[0,T],\label{eq:darcy_dim}\\
	\nabla_{x,z}\cdot u&=0,
	&&\text{in }\Omega(t)\times[0,T],\label{eq:incomp_dim}\\
	p&=\gamma\,\mathcal{E}_{\Gamma(t)}-\tau\,\mathcal{D}_{\Gamma(t)},
	&&\text{on }\Gamma(t)\times[0,T],\label{eq:dyn_dim}\\
	\partial_t h&=u\cdot N,
	&&\text{on }\Gamma(t)\times[0,T],\label{eq:kin_dim}\\
	u_3&=0,
	&&\text{on }\Gamma_{\mathrm{bot}}\times[0,T].\label{eq:bot_dim}
\end{align}
Here $\mu>0$ denotes the fluid viscosity, $\kappa>0$ the permeability of the porous medium,
$\rho>0$ the density, and $G>0$ the gravitational acceleration.
The parameter $\chi\in\{\pm1\}$ encodes the Rayleigh--Taylor sign:
$\chi=1$ corresponds to the gravitationally stable configuration,
while $\chi=-1$ describes the unstable regime. The dynamic boundary condition involves an elastic forcing of Willmore type,
together with a dissipative correction.
More precisely, the elastic operator $\mathcal{E}_{\Gamma(t)}$ depends on the geometry
of the interface through its mean and Gauss curvatures,
while the dissipative term $\mathcal{D}_{\Gamma(t)}$ accounts for tangential diffusion
along the interface
\footnote{Their explicit expressions, as well as the associated geometric quantities,
are recalled in Section~\ref{sec:elastic_muskat}.}.  This formulation couples a bulk Darcy flow with higher-order geometric forces acting on the interface,
and provides a natural framework to investigate the interplay between gravity,
tangential diffusion and elasticity in porous media flows.

\subsection*{Previous results}
The Muskat problem is a classical free-boundary model for flows in porous media and has been widely studied due to its
applicability and the delicate interplay between geometry, stability, and regularization mechanisms; see \cite{Bear1988} and
the references therein. The qualitative behavior depends strongly on surface tension, gravity, and whether one considers a
one-phase or two-phase configuration.

With surface tension and without gravity ($g=0$), the capillary term yields a parabolic regularization. Local-in-time
existence for large data was first proved in \cite{DuchonRobert1984,Chen1993,EscherSimonett1997}, and \cite{HQNguyen2019}
established well-posedness for initial data at any subcritical Sobolev regularity. For small perturbations, global existence
near equilibria (planar interfaces or circles) follows from the dominance of the parabolic smoothing over nonlinear effects;
see \cite{Chen1993,ConstantinPugh1993,YeTanveer2011}. When gravity is included, the Rayleigh--Taylor unstable configuration may
exhibit short-time instabilities for large data \cite{GHS2007}, and fingering phenomena can occur when a less viscous fluid
pushes a more viscous one, especially for small surface tension \cite{Otto1997,EscherMatioc2011}. Nonetheless, recent works
prove global existence and instantaneous smoothing for near-planar solutions even in gravity-unstable regimes \cite{GG-BS2020},
and bubble-type solutions can remain globally regular under suitable smallness assumptions on the slope relative to surface
tension \cite{GG-JPS2019Bubble}.

Without surface tension, the dynamics changes substantially. In non-graph settings, finite-time singularities may occur via
particle collision mechanisms \cite{CCFG2016,CordobaPernas-Castano2017}. In contrast, in the stable graph regime there are
maximum principles controlling the height and slope \cite{Kim2003,AlazardOneFluid2019}, which have been used to prove
well-posedness for arbitrary slopes \cite{DGN2021}. At the same time, solutions do not generally smooth instantly
\cite{APW2023}. For small slopes, one has instantaneous smoothing in the stable regime and ill-posedness in the unstable one
\cite{GG-JPS2019}. Moreover, in the stable case, solutions with surface tension converge to solutions without surface tension
as the capillarity coefficient vanishes \cite{Ambrose2014}.

Several features above are specific to the one-phase problem. In the two-phase case, particle collision does not occur
\cite{GancedoStrain2014}, although other singular behaviors can arise, such as loss of the graph property in finite time
\cite{CCFGL-F2012}. We refer to \cite{Sema2017} for a survey and to \cite{GancedoLazar2022,AlazardHung2023} for
recent developments. In the two-dimensional setting,
Wan and Yang \cite{WanYang2026} introduced and studied a Muskat-type formulation in which the free boundary is coupled to an
elastic restoring mechanism, and they established a well-posedness theory in that framework. \medskip

Thin-film and lubrication approximations form another closely related literature. The classical thin-film equation is a
degenerate fourth-order parabolic model, with global weak solutions and singularity formation studied in
\cite{BF1990,BBD-P1995,BP1996} and \cite{Con1993,Con2018}. Thin-film systems for porous media and multiple layers were derived
from Muskat-type models in \cite{EMM2012}, including capillary and gravitational effects. For purely gravity-driven motion,
local existence and stability results were obtained in \cite{EMM2012}, while global weak solutions and convergence were proved
in \cite{ELM2011}; purely capillarity-driven dynamics was treated in \cite{Mat2012}. In the combined gravity--capillarity
regime, local existence and stability in the stable case were shown in \cite{EM2013}, and global weak solutions via
gradient-flow methods were obtained in \cite{LM2013,LM2014}, with periodic low-regularity results in \cite{BG-B2019}. More generally, asymptotic model hierarchies for free-boundary flows in porous media, including weakly nonlinear regimes, have been developed in \cite{GraneroBelinchonScrobogna2019,GraneroBelinchonScrobogna2020}. A recent rigorous derivation of thin-film models directly from the one-phase Muskat problem in an unstable regime is provided by Bocchi and Gancedo \cite{BocchiGancedo2022}.
\medskip

Finally, the present work connects with fluid--structure interaction models in which an elastic interface is coupled to the
surrounding flow. Related problems include hydroelastic waves, the Peskin model, and viscoelastic filtration; see Cameron and
Strain \cite{MR4673875}, Gahn \cite{MR4959949}, Meirmanov \cite{MR2863467}, and Plotnikov and Toland \cite{MR2812947}. In the
hydroelastic setting the fluid is typically governed by the incompressible Euler equations; we refer to \cite{MR2812947} and
the references therein, as well as to \cite{MR2812939} for numerical and experimental developments. Local well-posedness for
two-dimensional hydroelastic waves was established in \cite{MR3656704,MR3608168}, with vorticity and higher-dimensional
extensions in \cite{MR4104949}. Closer in spirit to the present paper, Muskat models with elastic interfaces have been studied very recently by Wan and Yang
\cite{WanYang2026}. More precisely, they consider a two-dimensional Muskat system coupled to an elastic restoring mechanism
and establish a well-posedness theory for the resulting model.
\subsection*{Main results}
The aim of this paper is twofold. First, starting from an elastic one-phase Muskat free boundary problem in a bounded-depth porous medium, we derive reduced interface equations in asymptotic regimes where the dynamics is expected to simplify. Second, for the reduced models that arise, we develop a Wiener-space well-posedness theory that is robust enough to capture the quasilinear/nonlocal structure inherited from the Dirichlet--to--Neumann map and the elastic/dissipative effects. In particular, our analysis produces global solutions for small data and quantitative decay in Wiener norms.
\vspace{-1cm}
\subsubsection*{Weakly nonlinear small-slope models}
In the order-one depth regime ($\delta=1$) and for small interface steepness $\sigma\ll1$, we expand the Dirichlet--to--Neumann operator around the flat equilibrium and retain the evolution up to quadratic order in $\sigma$. Restricting to one space dimension for clarity, we obtain two weakly nonlinear nonlocal models. The first one keeps the quadratic contribution involving $\partial_t h$ inside the nonlinearity and reads
\textcolor{black}{
\begin{align*}
	\big(1-\Theta\,\mathcal G_0\,\partial_{xx}\big)\partial_t h
	&= -\chi\,\mathcal G_0\,h
	-\frac{\lambda}{4}\,\mathcal G_0\,\partial_{xxxx} h \notag\\
	&\quad+\sigma\chi\Big(\mathcal G_0\big(h\,\Lambda\tanh(\Lambda)\,h\big)
	+\partial_x\big(h\,\partial_x h\big)\Big)\notag\\
	&\quad+\sigma\frac{\lambda}{4}\Big(\mathcal G_0\big(h\,\mathcal G_0\,\partial_{xxxx} h\big)
	+\partial_x\big(h\,\partial_{xxxxx} h\big)\Big)\notag\\
	&\quad-\sigma\Theta\Big(\mathcal G_0\big(h\,\mathcal G_0\,\partial_{xx}\partial_t h\big)
	+\partial_x\big(h\,\partial_{xxx}\partial_t h\big)\Big),
\end{align*}
where $\mathcal G_0$ denotes the Dirichlet-to-Neumann map. Similarly, the previous equation, after making explicit the Fourier multiplier corresponding to the Dirichlet-to-Neumann map, reduces to}
\begin{align}
	\label{eq:main_model1}
	\big(1-\Theta\,\Lambda\tanh(\Lambda)\,\partial_{xx}\big)\partial_t h
	&= -\chi\,\Lambda\tanh(\Lambda)\,h
	-\frac{\lambda}{4}\,\Lambda\tanh(\Lambda)\,\partial_{xxxx} h \notag\\
	&\quad+\sigma\chi\Big(\Lambda\tanh(\Lambda)\big(h\,\Lambda\tanh(\Lambda)\,h\big)
	+\partial_x\big(h\,\partial_x h\big)\Big)\notag\\
	&\quad+\sigma\frac{\lambda}{4}\Big(\Lambda\tanh(\Lambda)\big(h\,\Lambda\tanh(\Lambda)\,\partial_{xxxx} h\big)
	+\partial_x\big(h\,\partial_{xxxxx} h\big)\Big)\notag\\
	&\quad-\sigma\Theta\Big(\Lambda\tanh(\Lambda)\big(h\,\Lambda\tanh(\Lambda)\,\partial_{xx}\partial_t h\big)
	+\partial_x\big(h\,\partial_{xxx}\partial_t h\big)\Big).
\end{align}
A second approximation, valid at the same order, consists in replacing $\partial_t h$ inside the quadratic $\sigma\Theta$--term
by its leading-order expression; this yields the alternative weakly nonlinear model
\textcolor{black}{
\begin{align*}
	\big(1-\Theta\,\mathcal G_0\,\partial_{xx}\big)\partial_t h
	&= -\chi\,\mathcal G_0\,h
	-\frac{\lambda}{4}\,\mathcal G_0\,\partial_{xxxx} h \notag\\
	&\quad+\sigma\chi\Big(\mathcal G_0\big(h\,\Lambda\tanh(\Lambda)\,h\big)
	+\partial_x\big(h\,\partial_x h\big)\Big)\notag\\
	&\quad+\sigma\frac{\lambda}{4}\Big(\mathcal G_0\big(h\,\mathcal G_0\,\partial_{xxxx} h\big)
	+\partial_x\big(h\,\partial_{xxxxx} h\big)\Big)\notag\\
	&\quad-\sigma\Theta\Big(\mathcal G_0\big(h\,\mathcal G_0\,\partial_{xx}\mu\big)
	+\partial_x\big(h\,\partial_{xxx}\mu\big)\Big),
\end{align*}
where
\begin{equation*}
	\mu=\big(1-\Theta\,\mathcal G_0\,\partial_{xx}\big)^{-1}
	\left[-\chi\,\mathcal G_0\,h-\frac{\lambda}{4}\,\mathcal G_0\,\partial_{xxxx}h\right].
\end{equation*}
Again, making explicit the Fourier representation of the Dirichlet-to-Neumann map, the previous system reduces to}
\begin{align}
	\label{eq:main_model2}
	\big(1-\Theta\,\Lambda\tanh(\Lambda)\,\partial_{xx}\big)\partial_t h
	&= -\chi\,\Lambda\tanh(\Lambda)\,h
	-\frac{\lambda}{4}\,\Lambda\tanh(\Lambda)\,\partial_{xxxx} h \notag\\
	&\quad+\sigma\chi\Big(\Lambda\tanh(\Lambda)\big(h\,\Lambda\tanh(\Lambda)\,h\big)
	+\partial_x\big(h\,\partial_x h\big)\Big)\notag\\
	&\quad+\sigma\frac{\lambda}{4}\Big(\Lambda\tanh(\Lambda)\big(h\,\Lambda\tanh(\Lambda)\,\partial_{xxxx} h\big)
	+\partial_x\big(h\,\partial_{xxxxx} h\big)\Big)\notag\\
	&\quad-\sigma\Theta\Big(\Lambda\tanh(\Lambda)\big(h\,\Lambda\tanh(\Lambda)\,\partial_{xx}\mu\big)
	+\partial_x\big(h\,\partial_{xxx}\mu\big)\Big),
\end{align}
with
\begin{equation}
	\label{eq:main_mu}
	\mu=\big(1-\Theta\,\Lambda\tanh(\Lambda)\,\partial_{xx}\big)^{-1}
	\left[-\chi\,\Lambda\tanh(\Lambda)\,h-\frac{\lambda}{4}\,\Lambda\tanh(\Lambda)\,\partial_{xxxx}h\right].
\end{equation}
\textcolor{black}{The dimensionless parameters are defined in \eqref{eq:eps_delta_sigma_dim} and \eqref{eq:lambdaTheta_defs} below.}

The novelty here is that elasticity and tangential dissipation lead, after the DtN expansion, to nonlocal 
equations in which the time derivative is acted upon by a nonlocal operator and, in the first model, also appears inside a
quadratic nonlocal commutator. 
\smallskip

On the analytical side, we prove that both weakly nonlinear small-slope models \eqref{eq:main_model1} and
\eqref{eq:main_model2}--\eqref{eq:main_mu} are well-posed for small data in the Wiener framework. More precisely, working on
$\TT$ and for zero-mean initial data $h_0$, we show that smallness of $\|h_0\|_{A^1}$ yields the existence and uniqueness of
mild solutions. In the case $\lambda=0$ we obtain local well-posedness in $A^1$ for \eqref{eq:main_model1}, and if in
addition $h_0\in A^3$ \textcolor{black}{and
$\|h_0\|_{A^0}+\Theta\tanh(1)\|h_0\|_{A^3}$ is sufficiently small, then the solution is global and decays in Wiener norms (Theorem~\ref{thm:wp_A1}). For $\lambda>0$ we
establish global well-posedness in $A^3(\TT)$ under the same two smallness conditions together with exponential decay
in $A^0$ for \eqref{eq:main_model1} (Theorem~\ref{thm:wp_A3_lambda}). }We also treat the alternative approximation
\eqref{eq:main_model2}--\eqref{eq:main_mu}, where the absence of a nonlinear operator acting on $\partial_t h$ simplifies the
argument, and obtain the same global well-posedness and decay (Theorem~\ref{thm:wp_A3_model2}). Finally, all these results
extend to the infinite-depth analogues obtained by replacing $\Lambda\tanh(\Lambda)$ with $\Lambda$
(Remark~\ref{rem:infinite_depth}).

\medskip
\subsubsection*{A thin-film (lubrication) model}
We next consider the long-wave thin-film regime $\delta\ll1$. Rewriting the kinematic condition in conservative form and
flattening the moving domain onto a fixed strip, we derive a lubrication-type evolution with variable mobility. In one space
dimension, the resulting leading-order thin-film model reads
\begin{equation}
	\label{eq:main_lub}
	\partial_t h+\sqrt{\delta}\,\partial_x\Big((1+\varepsilon h)\,\partial_x\mu\Big)=0,
	\qquad
	\mu:=\Theta\,\partial_{xx}\partial_t h-\chi h-\frac{\lambda}{4}\,\partial_x^{4}h.
\end{equation}
The novelty of \eqref{eq:main_lub} is that the dissipation term $\Theta\,\partial_{xx}\partial_t h$ remains coupled to the
mobility $(1+\varepsilon h)$ at leading order, producing a genuinely nonlinear elliptic operator acting on $\partial_t h$
when the equation is written as a first-order system (\textcolor{black}{compare with} \cite{BG-B2019} \textcolor{black}{where this elliptic problem acting on the time derivative is absent}). \textcolor{black}{A similar lubrication approximation for Navier-Stokes flow was derived in \cite{BukalMuha2022} (see also \cite{BukalMuha2021}). The thin-film PDE derived in \cite{BukalMuha2022} also has an elliptic problem acting on the time derivative of the unknown.}

\smallskip
Finally, we prove a global well-posedness result for \eqref{eq:main_lub} in Wiener spaces (Theorem~\ref{thm:wp_lub_A4_statement}). More precisely, for zero-mean data
$h_0\in A^4(\TT)$ satisfying a smallness assumption in $A^1$, we construct a unique global mild solution
$h\in C([0,\infty);A^4(\TT))$. Moreover, the solution satisfies a dissipative energy inequality that yields uniform-in-time
bounds and decay in the Wiener norm $\|h(t)\|_{A^0}$ as $t\to\infty$.

\subsection*{Plan of the paper}
Section~\ref{sec:elastic_muskat} introduces the elastic Muskat problem, recalling the Eulerian formulation and geometric quantities (Subsection~\ref{subsec:model_eulerian}) and deriving the potential formulation together with the nondimensional system and parameters (Subsection~\ref{subsec:potential_formulation}). In Section~\ref{sec:DTN_sigma} we derive weakly nonlinear interface models in the small-slope regime via an expansion of the Dirichlet--to--Neumann operator. \textcolor{black}{In Section~\ref{sec:lubrication} we study the thin-film regime: we first rewrite the kinematic condition in flux form (Subsection~\ref{subsec:lubrication_flux}), then flatten the moving strip to a fixed domain (Subsection~\ref{subsec:lubrication_flattening}), and finally derive a leading-order lubrication model (Subsection~\ref{subsec:lubrication_model}).
Section~\ref{sec:notation_spaces} collects the notation, Wiener spaces, and the basic inequalities used throughout the paper.
Section~\ref{sec:wp1} is devoted to the well-posedness theory of these weakly nonlinear models in Wiener spaces. } Finally, Section~\ref{sec:wp2} establishes the well-posedness of the thin-film model in appropriate Wiener spaces.

\section{The elastic Muskat problem and its potential formulation}\label{sec:elastic_muskat}

In this section we make precise the elastic Muskat setting introduced in the introduction and fix notation.
Rather than repeating the full Eulerian formulation, we only recall the parts that will be used in the subsequent non--dimensionalization and in the derivation of a potential formulation on the moving domain.

\subsection{Geometric setting and eulerian formulation}\label{subsec:model_eulerian}

We work in a horizontally periodic geometry and write the free surface as a graph.
At time $t$ the fluid occupies
\[
\Omega(t)=\big\{(x,z)\in\mathbb{R}^3:\,-d<z<h(x,t)\big\},
\qquad x=(x_1,x_2)\in(-L\pi,L\pi)^2,
\]
with free boundary $\Gamma(t)=\{(x,h(x,t))\}$ and flat impermeable bottom
$\Gamma_{\mathrm{bot}}=\{(x,-d)\}$. We impose periodic boundary conditions in $x$.
The unknowns are the Darcy velocity $u=u(x,z,t)$, the pressure $p=p(x,z,t)$, and the interface height $h=h(x,t)$.

The governing equations are Darcy’s law with gravity, incompressibility, a dynamic boundary condition coupling elastic
and dissipative effects, the kinematic condition, and an impermeable bottom boundary; see
\eqref{eq:darcy_dim}--\eqref{eq:bot_dim} in the Introduction. We only recall here that the Rayleigh--Taylor sign
$\chi\in\{\pm1\}$ distinguishes the stable ($\chi=1$) and unstable ($\chi=-1$) configurations.

Since $\Gamma(t)$ is a graph, we use the (non-unit) upward normal
\begin{equation}\label{eq:nonunit_normal_dim}
	N:=(-\nabla_x h,\,1),
	\qquad |N|=\sqrt{1+|\nabla_x h|^2},
\end{equation}
and the unit normal $n:=N/|N|$. The dynamic boundary condition involves a Willmore-type elastic forcing
together with a tangential dissipative correction, written in \eqref{eq:dyn_dim} as
\[
p=\gamma\,\mathcal{E}_{\Gamma(t)}-\tau\,\mathcal{D}_{\Gamma(t)}\qquad\text{on }\Gamma(t).
\]

\textcolor{black}{
Indeed,as in Plotnikov and Toland \cite{MR2812947}, we assume that the elastic force derives from a curvature-based bending energy of the form
	\begin{equation}
	E_b[\eta]
	= \int_{L\mathbb{S}^{1}\times L\mathbb{S}^{1}}
	W_b\big(H(\eta)\big)\,
	\sqrt{1+|\nabla_{x}\eta|^2}\,dx,
	\label{eq:bending-general}
	\end{equation}
	where $H(\eta)$ denotes the mean curvature of the free surface $\Gamma(t)$ and $W_b$ is a prescribed bending density. A widely used choice is
	\[
	W_b(H)=\frac{1}{2}H^2.
	\]
The elastic force associated with this energy is obtained from the first	variation of the previous functional. As shown in \cite{MR2812947}, one obtains the fourth-order nonlinear geometric operator
\begin{equation}\label{eq:E_dim}
	\mathcal{E}_{\Gamma(t)}
	=\frac12\,\Delta_\Gamma \mathcal{H}(h)
	+\mathcal{H}(h)^3-\mathcal{H}(h)\,K(h),
\end{equation}
where $\Delta_\Gamma$ is the Laplace--Beltrami operator and $K$ is the	Gauss curvature of the evolving surface.  Finally, we also consider viscoelastic damping of Kelvin--Voigt type. This damping is modelled by
\begin{equation}\label{eq:D_dim}
	\mathcal{D}_{\Gamma(t)}=\Delta_x \partial_t h,
	\qquad \Delta_x=\partial_{x_1}^2+\partial_{x_2}^2.
\end{equation}
	corresponding to a Rayleigh dissipation potential $\mathcal{R}(\eta)=
	\frac{\gamma}{2}\int | \partial_{t}\nabla_{x}\eta|^2\,dx$. 
The presence of $\mathcal{D}_{\Gamma(t)}$ leads to a stabilization term in the wave equations derived below.   In the recent paper \cite{WanYang2026}, Wan and Yang
 study the corresponding elastic Muskat problem without
this additional damping, so we do not claim that it is required for
well-posedness of the full free-boundary system. In the reduced models derived
here, however, it produces the operator
$1-\Theta\mathcal G_0\partial_{xx}$ acting on $\partial_t h$, which is used
directly in the inverse and higher-Wiener-norm estimates below.}

For completeness, we recall the graph expressions used throughout the paper. The first fundamental form is
\begin{equation}\label{eq:metric_dim}
	g_{ij}=\delta_{ij}+h_{x_i}h_{x_j},
	\qquad |g|=1+|\nabla_x h|^2,
	\qquad g^{ij}=(g_{ij})^{-1},
\end{equation}
and the Laplace--Beltrami operator acting on a scalar $f=f(x)$ is
\begin{equation}\label{eq:LB_dim}
	\Delta_\Gamma f
	=\frac{1}{\sqrt{|g|}}\,
	\partial_{x_i}\Big(\sqrt{|g|}\,g^{ij}\,\partial_{x_j}f\Big),
\end{equation}
(with Einstein summation over $i,j\in\{1,2\}$). With the convention
$\mathcal H=\frac12(\kappa_1+\kappa_2)$, the mean curvature and Gauss curvature are
\begin{equation}\label{eq:H_dim}
	\mathcal{H}(h)
	=\frac{1}{2\sqrt{|g|}}
	\Big(g^{11}h_{x_1x_1}+2g^{12}h_{x_1x_2}+g^{22}h_{x_2x_2}\Big),
\end{equation}
\begin{equation}\label{eq:K_dim}
	K(h)=\frac{h_{x_1x_1}h_{x_2x_2}-(h_{x_1x_2})^{2}}{\big(1+|\nabla_x h|^{2}\big)^{2}}.
\end{equation}

\subsection{Potential formulation and dimensionless formulation}\label{subsec:potential_formulation}

A convenient way to reduce the bulk variables is to introduce a potential associated with Darcy’s law.
Following \eqref{eq:darcy_dim}, we define
\begin{equation}\label{eq:Phi_def_dim}
	\Phi:=\frac{\kappa}{\mu}\big(-p-\chi\rho G z\big),
\end{equation}
so that
\begin{equation}\label{eq:u_gradPhi_dim}
	u=\nabla_{x,z}\Phi\qquad \text{in }\Omega(t).
\end{equation}
By incompressibility \eqref{eq:incomp_dim}, $\Phi$ is harmonic in the bulk,
\begin{equation}\label{eq:LaplacePhi_dim}
	\Delta_{x,z}\Phi=0 \qquad \text{in }\Omega(t),
\end{equation}
and the impermeability condition \eqref{eq:bot_dim} becomes the Neumann condition
\begin{equation}\label{eq:bottom_Neumann_dim}
	\partial_z\Phi=0 \qquad \text{on }\Gamma_{\mathrm{bot}}.
\end{equation}
On the free boundary, the dynamic condition \eqref{eq:dyn_dim} yields the Dirichlet \textcolor{black}{boundary condition}
\begin{equation}\label{eq:top_Dirichlet_dim}
	\Phi
	=-\frac{\gamma\kappa}{\mu}\,\mathcal{E}_{\Gamma(t)}
	+\frac{\tau\kappa}{\mu}\,\mathcal{D}_{\Gamma(t)}
	-\chi\frac{\kappa\rho G}{\mu}\,h,
	\qquad \text{on }\Gamma(t).
\end{equation}
Hence, for each fixed $t$, $\Phi$ solves the mixed boundary value problem
\begin{equation}\label{eq:Phi_elliptic_dim}
	\left\lbrace
	\begin{aligned}
		&\Delta_{x,z}\Phi=0, &&\text{in }\Omega(t),\\
		&\Phi=-\frac{\gamma\kappa}{\mu}\,\mathcal{E}_{\Gamma(t)}
		+\frac{\tau\kappa}{\mu}\,\mathcal{D}_{\Gamma(t)}
		-\chi\frac{\kappa\rho G}{\mu}\,h, &&\text{on }\Gamma(t),\\
		&\partial_z\Phi=0, &&\text{on }\Gamma_{\mathrm{bot}}.
	\end{aligned}
	\right.
\end{equation}
The kinematic condition \eqref{eq:kin_dim} can be written in terms of $\Phi$ using $u=\nabla_{x,z}\Phi$ and \eqref{eq:nonunit_normal_dim}:
\begin{equation}\label{eq:kinematic_Phi_dim}
	\partial_t h
	=\nabla_{x,z}\Phi\cdot N
	=\partial_z\Phi-\nabla_x h\cdot \nabla_x\Phi
	\qquad \text{on }z=h(x,t).
\end{equation}
Equations \eqref{eq:Phi_elliptic_dim}--\eqref{eq:kinematic_Phi_dim} thus provide a closed evolution:
given $h(t)$, solve \eqref{eq:Phi_elliptic_dim} for $\Phi(t)$ in $\Omega(t)$ and update the interface by \eqref{eq:kinematic_Phi_dim}.  \medskip

We now introduce a nondimensional formulation of the elastic Muskat system, which allows us to identify the relevant dimensionless parameters and to prepare the asymptotic analysis carried out in the subsequent sections. To that purpose, we introduce dimensionless variables
\begin{equation}\label{eq:nondim_vars_dim}
	(x_1,x_2)=L(\tilde x_1,\tilde x_2),\qquad
	z=d\,\tilde z,\qquad
	t=\frac{\mu L}{\rho\kappa G}\,\tilde t,
\end{equation}
and scale the unknowns as
\begin{equation}\label{eq:nondim_unknowns_dim}
	h(x,t)=H\,\tilde h(\tilde x,\tilde t),
	\qquad
	\Phi(x,z,t)=\frac{H\kappa\rho G}{\mu}\,\tilde\Phi(\tilde x,\tilde z,\tilde t).
\end{equation}
We define the non-dimensional parameters
\begin{equation}\label{eq:eps_delta_sigma_dim}
	\delta=\frac{d^2}{L^2},\qquad
	\varepsilon=\frac{H}{d},\qquad
	\sigma=\frac{H}{L}=\varepsilon\sqrt{\delta}.
\end{equation}
Here $\delta$ is the depth-to-wavelength aspect ratio, $\varepsilon=H/d$ is the amplitude-to-depth ratio,
and $\sigma=H/L$ is the geometric slope parameter. In the rescaled variables the horizontal periodic box becomes $(-\pi,\pi)^2$.
Dropping tildes from now on, we write $\TT^2:=(-\pi,\pi)^2$ and the nondimensional fluid domain and boundaries are
\begin{align}\label{eq:OmegaGamma_nondim}
	\Omega(t)
	&=\Big\{(x,z)\in\TT^2\times\RR:\ -1<z<\varepsilon h(x,t)\Big\},\\
	\Gamma(t)
	&=\Big\{(x,\varepsilon h(x,t)):\ x\in\TT^2\Big\},\\
	\Gamma_{\mathrm{bot}}
	&=\TT^2\times\{-1\}.
\end{align}
In these variables, the potential $\Phi$ satisfies
\begin{equation}\label{eq:final_system_potential}
	\left\lbrace
	\begin{aligned}
		&\left(\delta\Delta_x+\partial_{zz}\right)\Phi=0,
		&&\text{in }\Omega(t),\\[1mm]
		&\Phi=-\chi h-\lambda\,\mathcal E^\sigma(h)+\Theta\,\Delta_x\partial_t h,
		&&\text{on }\Gamma(t),\\[1mm]
		&\partial_z\Phi=0,
		&&\text{on }\Gamma_{\mathrm{bot}},\\[1mm]
		&\partial_t h
		=\frac{1}{\sqrt{\delta}}\,\partial_z\Phi-\sigma\,\nabla_x h\cdot \nabla_x\Phi,
		&&\text{on }\Gamma(t).
	\end{aligned}
	\right.
\end{equation}
Here
\begin{equation}\label{eq:lambdaTheta_defs}
	\lambda:=\frac{\gamma}{2\rho G L^4},
	\qquad
	\Theta:=\frac{\tau\kappa}{\mu L^3},
\end{equation}
and the geometric aspect ratio $\sigma=\varepsilon\sqrt{\delta}$ enters only through the elastic operator $\mathcal E^\sigma$. More precisely, we define the induced metric (dimensionless) by
\begin{equation}\label{eq:geom_metric_sigma}
	g_{ij}=\delta_{ij}+\sigma^2 h_{x_i}h_{x_j},
	\qquad
	|g|=1+\sigma^2|\nabla_x h|^2,
	\qquad
	g^{ij}=(g_{ij})^{-1}.
\end{equation}
The upward unit normal is
\begin{equation}\label{eq:unit_normal_sigma}
	n=\frac{1}{\sqrt{|g|}}\big(-\sigma\nabla_x h,\,1\big).
\end{equation}
With these definitions, the nondimensional Willmore-type elastic operator in terms of $g^{ij}$ and $|g|$ is given by
\begin{align}\label{eq:E_sigma_explicit}
	\mathcal E^\sigma(h)
	&=
\,\frac{1}{\sqrt{|g|}}\,
	\partial_{x_i}\!\left(
	\sqrt{|g|}\,g^{ij}\,\partial_{x_j}\left(
	\frac{1}{2\sqrt{|g|}}
	\Big(
	g^{11}h_{x_1x_1}
	+2g^{12}h_{x_1x_2}
	+g^{22}h_{x_2x_2}
	\Big)\right)\right)\notag\\
	&\quad+\frac{\sigma^2}{4|g|^{3/2}}
	\Big(
	g^{11}h_{x_1x_1}
	+2g^{12}h_{x_1x_2}
	+g^{22}h_{x_2x_2}
	\Big)^3\notag\\
	&\quad-\frac{\sigma^2}{\sqrt{|g|}}
	\Big(
	g^{11}h_{x_1x_1}
	+2g^{12}h_{x_1x_2}
	+g^{22}h_{x_2x_2}
	\Big)\,
	\frac{
		h_{x_1x_1}h_{x_2x_2}-(h_{x_1x_2})^{2}}
	{\big(1+\sigma^2|\nabla_x h|^{2}\big)^{2}}.
\end{align}

	\section{A weakly nonlinear model in the small-slope regime }\label{sec:DTN_sigma}
In this section we derive a weakly nonlinear, nonlocal evolution equation by expanding the
Dirichlet--to--Neumann (DtN) operator in the small-slope regime. The argument follows closely the DtN expansion strategy
introduced in \cite{GraneroBelinchonScrobogna2019}, adapted to the present elastic setting. The expansion parameter is the geometric
aspect ratio $\sigma:=H/L$, so that typical interface slopes satisfy $|\nabla_x h_{\mathrm{phys}}|\sim\sigma$.
We develop the DtN map and the resulting evolution model with an accuracy up to $\mathcal{O}(\sigma^{2})$.
Throughout, we restrict to the order-one depth regime $\delta=1$, in which case $\sigma=\varepsilon\sqrt{\delta}=\varepsilon$.
Let us stress that this small-slope expansion is conceptually different from the lubrication limit $\delta\ll1$ considered later.

We start from the nondimensional potential formulation \eqref{eq:final_system_potential} obtained in the previous section
and specialize to $\delta=1$. For each fixed $t$, the potential $\Phi$ solves
\[
\Delta_{x,z}\Phi=0\quad\text{in }\Omega(t),\qquad
\partial_z\Phi=0\quad\text{on }\Gamma_{\mathrm{bot}},
\]

On the free surface $\Gamma(t)$ we set
\begin{equation}\label{eq:Psi_def_sigma_full}
	\Psi:=-\chi h-\lambda\,\mathcal E^{\sigma}(h)+\Theta\,\Delta_x\partial_t h,
	\qquad \chi\in\{\pm1\},
\end{equation}
where $\mathcal E^{\sigma}$ is the elastic operator defined in \eqref{eq:E_sigma_explicit}. Then the kinematic
condition in \eqref{eq:final_system_potential} can be written in terms of the Dirichlet--to--Neumann map as
\begin{equation}\label{eq:kinematic_DTN_sigma_full}
	\partial_t h=\mathcal G(\sigma h)\Psi.
\end{equation}
For the flat strip ($h\equiv0$), the DtN operator is the Fourier multiplier
\begin{equation}\label{eq:G0_sigma_full}
	\mathcal G_0:=\Lambda\tanh(\Lambda),
	\qquad
	\widehat{\Lambda f}(\xi)=|\xi|\,\widehat f(\xi),
\end{equation}
where $\xi\in\ZZ^2$ (periodic setting). Moreover, using an expansion of the DtN operator in terms of the steepness parameter $\sigma$ (cf. \cite[Section 3.6.2]{Lannes2013}) around the rest state we find that 
\begin{equation}\label{eq:DTN_expansion_sigma_full}
	\mathcal G(\sigma h)\phi
	=\mathcal G_0\phi
	-\sigma\Big(\mathcal G_0\big(h\,\mathcal G_0\phi\big)+\nabla_x\cdot\big(h\,\nabla_x\phi\big)\Big)
	+\mathcal{O}(\sigma^2).
\end{equation}
Substituting $\phi=\Psi$ and using \eqref{eq:kinematic_DTN_sigma_full} yields
\begin{equation}\label{eq:model_sigma_pre_full}
	\partial_t h
	=\mathcal G_0\Psi
	-\sigma\Big(\mathcal G_0\big(h\,\mathcal G_0\Psi\big)+\nabla_x\cdot\big(h\,\nabla_x\Psi\big)\Big)
	+\mathcal{O}(\sigma^2).
\end{equation}
To obtain a model consistent up to $\mathcal{O}(\sigma^{2})$, it suffices to approximate the elastic operator by its leading term
\begin{equation*}
	\mathcal E^\sigma(h)=\frac14\,\Delta_x^2 h+\mathcal{O}(\sigma^2),
\end{equation*}
so that
\begin{equation}\label{eq:Psi_linear_sigma_full}
	\Psi=-\chi h-\frac{\lambda}{4}\Delta_x^2 h+\Theta\,\Delta_x\partial_t h+\mathcal{O}(\sigma^2).
\end{equation}

To derive a weakly nonlinear model accurate up to $\mathcal{O}(\sigma^2)$, we substitute \eqref{eq:Psi_linear_sigma_full} into \eqref{eq:model_sigma_pre_full} and retain terms up to $\mathcal{O}(\sigma^2)$. More precisely, we find that
\begin{align}\label{eq:model_sigma_explicit_full}
	\partial_t h
	&= -\chi\,\mathcal G_0 h-\frac{\lambda}{4}\,\mathcal G_0\Delta_x^2 h+\Theta\,\mathcal G_0\Delta_x\partial_t h +\sigma\chi\Big(\mathcal G_0\big(h\,\mathcal G_0 h\big)+\nabla_x\cdot\big(h\,\nabla_x h\big)\Big)\notag\\
	&\quad+\sigma\frac{\lambda}{4}\Big(\mathcal G_0\big(h\,\mathcal G_0\Delta_x^2 h\big)+\nabla_x\cdot\big(h\,\nabla_x\Delta_x^2 h\big)\Big)-\sigma\Theta\Big(\mathcal G_0\big(h\,\mathcal G_0\Delta_x\partial_t h\big)+\nabla_x\cdot\big(h\,\nabla_x\Delta_x\partial_t h\big)\Big).
\end{align}
Equivalently, collecting the linear $\partial_t h$ terms on the left and recalling the definition of the Fourier multiplier \eqref{eq:G0_sigma_full} we find that
\begin{align}\label{eq:model_sigma_LHS_full_multiplier}
	\big(1-\Theta\,\Lambda\tanh(\Lambda)\,\Delta_x\big)\partial_t h
	&= -\chi\,\Lambda\tanh(\Lambda)\,h
	-\frac{\lambda}{4}\,\Lambda\tanh(\Lambda)\,\Delta_x^2 h \notag\\
	&\quad+\sigma\chi\Big(\Lambda\tanh(\Lambda)\big(h\,\Lambda\tanh(\Lambda)\,h\big)
	+\nabla_x\cdot\big(h\,\nabla_x h\big)\Big)\notag\\
	&\quad+\sigma\frac{\lambda}{4}\Big(\Lambda\tanh(\Lambda)\big(h\,\Lambda\tanh(\Lambda)\,\Delta_x^2 h\big)
	+\nabla_x\cdot\big(h\,\nabla_x\Delta_x^2 h\big)\Big)\notag\\
	&\quad-\sigma\Theta\Big(\Lambda\tanh(\Lambda)\big(h\,\Lambda\tanh(\Lambda)\,\Delta_x\partial_t h\big)
	+\nabla_x\cdot\big(h\,\nabla_x\Delta_x\partial_t h\big)\Big).
\end{align}
We observe that the previous equation implies
$$
\partial_t h=\big(1-\Theta\,\Lambda\tanh(\Lambda)\,\Delta_x\big)^{-1}\left[-\chi\,\Lambda\tanh(\Lambda)\,h
	-\frac{\lambda}{4}\,\Lambda\tanh(\Lambda)\,\Delta_x^2 h\right] + \mathcal{O}(\sigma).
$$
As a consequence, we find the second weakly nonlinear model
\begin{align}\label{eq:model_sigma_LHS_full_multiplier2}
	\big(1-\Theta\,\Lambda\tanh(\Lambda)\,\Delta_x\big)\partial_t h
	&= -\chi\,\Lambda\tanh(\Lambda)\,h
	-\frac{\lambda}{4}\,\Lambda\tanh(\Lambda)\,\Delta_x^2 h \notag\\
	&\quad+\sigma\chi\Big(\Lambda\tanh(\Lambda)\big(h\,\Lambda\tanh(\Lambda)\,h\big)
	+\nabla_x\cdot\big(h\,\nabla_x h\big)\Big)\notag\\
	&\quad+\sigma\frac{\lambda}{4}\Big(\Lambda\tanh(\Lambda)\big(h\,\Lambda\tanh(\Lambda)\,\Delta_x^2 h\big)
	+\nabla_x\cdot\big(h\,\nabla_x\Delta_x^2 h\big)\Big)\notag\\
	&\quad-\sigma\Theta\Big(\Lambda\tanh(\Lambda)\big(h\,\Lambda\tanh(\Lambda)\,\Delta_x\mu\big)
	+\nabla_x\cdot\big(h\,\nabla_x\Delta_x\mu\big)\Big),
\end{align}
with
\begin{equation}\label{eq:model_sigma_LHS_full_multiplier2b}
\mu=\big(1-\Theta\,\Lambda\tanh(\Lambda)\,\Delta_x\big)^{-1}\left[-\chi\,\Lambda\tanh(\Lambda)\,h
	-\frac{\lambda}{4}\,\Lambda\tanh(\Lambda)\,\Delta_x^2 h\right].
\end{equation}
\subsubsection*{One dimensional weakly nonlinear interface model}
	For a one-dimensional interface $h=h(x,t)$, $x\in\TT$, the finite-depth model \eqref{eq:model_sigma_LHS_full_multiplier} becomes
	\begin{align}\label{eq:model_sigma_LHS_full_multiplier_1D}
		\big(1-\Theta\,\Lambda\tanh(\Lambda)\,\partial_{xx}\big)\partial_t h
		&= -\chi\,\Lambda\tanh(\Lambda)\,h
		-\frac{\lambda}{4}\,\Lambda\tanh(\Lambda)\,\partial_{xxxx} h \notag\\
		&\quad+\sigma\chi\Big(\Lambda\tanh(\Lambda)\big(h\,\Lambda\tanh(\Lambda)\,h\big)
		+\partial_x\big(h\,\partial_x h\big)\Big)\notag\\
		&\quad+\sigma\frac{\lambda}{4}\Big(\Lambda\tanh(\Lambda)\big(h\,\Lambda\tanh(\Lambda)\,\partial_{xxxx} h\big)
		+\partial_x\big(h\,\partial_{xxxxx} h\big)\Big)\notag\\
		&\quad-\sigma\Theta\Big(\Lambda\tanh(\Lambda)\big(h\,\Lambda\tanh(\Lambda)\,\partial_{xx}\partial_t h\big)
		+\partial_x\big(h\,\partial_{xxx}\partial_t h\big)\Big).
	\end{align}
As before, we observe that the previous equation implies
$$
\partial_t h
		= \left(1-\Theta\,\Lambda\tanh(\Lambda)\,\partial_{xx}\right)^{-1}\left[-\chi\,\Lambda\tanh(\Lambda)\,h
		-\frac{\lambda}{4}\,\Lambda\tanh(\Lambda)\,\partial_{xxxx} h\right]+\mathcal{O}(\sigma).
$$
Hence, we find that the second weakly nonlinear model \eqref{eq:model_sigma_LHS_full_multiplier2}-\eqref{eq:model_sigma_LHS_full_multiplier2b} in the case of a one dimensional interface reads
	\begin{align}\label{eq:model_sigma_LHS_full_multiplier_1D2}
		\big(1-\Theta\,\Lambda\tanh(\Lambda)\,\partial_{xx}\big)\partial_t h
		&= -\chi\,\Lambda\tanh(\Lambda)\,h
		-\frac{\lambda}{4}\,\Lambda\tanh(\Lambda)\,\partial_{xxxx} h \notag\\
		&\quad+\sigma\chi\Big(\Lambda\tanh(\Lambda)\big(h\,\Lambda\tanh(\Lambda)\,h\big)
		+\partial_x\big(h\,\partial_x h\big)\Big)\notag\\
		&\quad+\sigma\frac{\lambda}{4}\Big(\Lambda\tanh(\Lambda)\big(h\,\Lambda\tanh(\Lambda)\,\partial_{xxxx} h\big)
		+\partial_x\big(h\,\partial_{xxxxx} h\big)\Big)\notag\\
		&\quad-\sigma\Theta\Big(\Lambda\tanh(\Lambda)\big(h\,\Lambda\tanh(\Lambda)\,\partial_{xx}\mu\big)
		+\partial_x\big(h\,\partial_{xxx}\mu\big)\Big)
	\end{align}
with
\begin{equation}\label{eq:model_sigma_LHS_full_multiplier_1D2b}
\mu	= \big(1-\Theta\,\Lambda\tanh(\Lambda)\,\partial_{xx}\big)^{-1}\left[-\chi\,\Lambda\tanh(\Lambda)\,h
		-\frac{\lambda}{4}\,\Lambda\tanh(\Lambda)\,\partial_{xxxx} h\right].
\end{equation}

\begin{remark}[Infinite depth models]
If instead of the bounded strip one considers the unbounded-in-$z$ geometry, then $\tanh(\Lambda)$ is replaced by $1$ and
$\mathcal G_0=\Lambda$. In that case the model \eqref{eq:model_sigma_explicit_full} becomes
\begin{align*}
	\partial_t h
	&= -\chi\,\Lambda h-\frac{\lambda}{4}\,\Lambda\Delta_x^2 h+\Theta\,\Lambda\Delta_x\partial_t h +\sigma\chi\Big(\Lambda\big(h\,\Lambda h\big)+\nabla_x\cdot\big(h\,\nabla_x h\big)\Big)\notag\\
	&\quad+\sigma\frac{\lambda}{4}\Big(\Lambda\big(h\,\Lambda\Delta_x^2 h\big)+\nabla_x\cdot\big(h\,\nabla_x\Delta_x^2 h\big)\Big)-\sigma\Theta\Big(\Lambda\big(h\,\Lambda\Delta_x\partial_t h\big)+\nabla_x\cdot\big(h\,\nabla_x\Delta_x\partial_t h\big)\Big).
\end{align*}
Therefore, for a one-dimensional interface $h=h(x,t)$, $x\in\TT$, the infinite-depth model reads
\begin{align*}
	\partial_t h
	&= -\chi\,\Lambda h-\frac{\lambda}{4}\,\Lambda\partial_{xxxx} h+\Theta\,\Lambda\partial_{xx}\partial_t h+\sigma\chi\Big(\Lambda\big(h\,\Lambda h\big)+\partial_x\big(h\,\partial_x h\big)\Big)\notag\\
	&\quad+\sigma\frac{\lambda}{4}\Big(\Lambda\big(h\,\Lambda\partial_{xxxx} h\big)+\partial_x\big(h\,\partial_{xxxxx} h\big)\Big)-\sigma\Theta\Big(\Lambda\big(h\,\Lambda\partial_{xx}\partial_t h\big)+\partial_x\big(h\,\partial_{xxx}\partial_t h\big)\Big)
\end{align*}
Moreover, proceeding  as above, we can also derive a different weakly nonlinear approximation equation at the same order of precision that given by
\begin{align}\label{eq:model_sigma_infinite_depth_full2}
	\partial_t h
	&= -\chi\,\Lambda h-\frac{\lambda}{4}\,\Lambda\Delta_x^2 h+\Theta\,\Lambda\Delta_x\partial_t h +\sigma\chi\Big(\Lambda\big(h\,\Lambda h\big)+\nabla_x\cdot\big(h\,\nabla_x h\big)\Big)\notag\\
	&\quad+\sigma\frac{\lambda}{4}\Big(\Lambda\big(h\,\Lambda\Delta_x^2 h\big)+\nabla_x\cdot\big(h\,\nabla_x\Delta_x^2 h\big)\Big)-\sigma\Theta\Big(\Lambda\big(h\,\Lambda\Delta_x\mu\big)+\nabla_x\cdot\big(h\,\nabla_x\Delta_x\mu\big)\Big).
\end{align}
with
\begin{equation}\label{eq:model_sigma_infinite_depth_full2b}
\mu=(1-\Theta\,\Lambda\Delta_x)^{-1} \left(-\chi\,\Lambda h-\frac{\lambda}{4}\,\Lambda\Delta_x^2 h\right).
\end{equation}
If the interface is one-dimensional, $h=h(x,t)$ with $x\in\TT$, then $\nabla_x=\partial_x$ and $\Delta_x=\partial_{xx}$, and
\eqref{eq:model_sigma_infinite_depth_full2}-\eqref{eq:model_sigma_infinite_depth_full2b} becomes
\begin{align*}
	\partial_t h
	&= -\chi\,\Lambda h-\frac{\lambda}{4}\,\Lambda\partial_{xxxx} h+\Theta\,\Lambda\partial_{xx}\partial_t h
	+\sigma\chi\Big(\Lambda\big(h\,\Lambda h\big)+\partial_x\big(h\,\partial_x h\big)\Big)\notag\\
	&\quad+\sigma\frac{\lambda}{4}\Big(\Lambda\big(h\,\Lambda\partial_{xxxx} h\big)+\partial_x\big(h\,\partial_{xxxxx} h\big)\Big)
	-\sigma\Theta\Big(\Lambda\big(h\,\Lambda\partial_{xx}\mu\big)+\partial_x\big(h\,\partial_{xxx}\mu\big)\Big),
\end{align*}
with
\begin{equation*}
	\mu=\bigl(1-\Theta\,\Lambda\partial_{xx}\bigr)^{-1}
	\left(-\chi\,\Lambda h-\frac{\lambda}{4}\,\Lambda\partial_{xxxx} h\right).
\end{equation*}
\end{remark}

	\section{A lubrication approximation in the thin-film regime}\label{sec:lubrication}
In this section we derive a lubrication (thin--film) model in the long-wave regime $\delta\ll1$.
Throughout, the asymptotic parameter is $\delta$ (the depth-to-wavelength ratio), while the geometric steepness
$\sigma=\varepsilon\sqrt{\delta}$ is not treated as an independent small parameter: it only enters through its
$\delta$--dependence (with $\varepsilon=H/d$ regarded as order one).
Our derivation follows closely the strategy introduced by Bocchi and Gancedo \cite{BocchiGancedo2022},
adapted here to the presence of elastic and dissipative effects.
Starting from the nondimensional potential formulation \eqref{eq:final_system_potential}, we first rewrite the kinematic
condition in conservative (flux) form, which is well suited for long-wave asymptotics. We then flatten the moving domain
onto a fixed strip and perform a $\delta$--expansion of the pullback potential, leading to a closed evolution equation for
the interface height $h$.

\subsection{Flux--potential formulation}\label{subsec:lubrication_flux}

Since $\Phi$ solves
\[
(\delta\Delta_x+\partial_{zz})\Phi=0 \qquad \text{in }\Omega(t),
\]
we have
\begin{equation}\label{eq:lub_dzz_identity}
	\partial_{zz}\Phi=-\delta\,\Delta_x\Phi.
\end{equation}
Integrating \eqref{eq:lub_dzz_identity} in $z\in(-1,\varepsilon h(x,t))$ and using the bottom condition
$\partial_z\Phi(x,-1,t)=0$ gives
\begin{equation}\label{eq:lub_dz_representation}
	\partial_z\Phi(x,\varepsilon h,t)
	=-\delta\int_{-1}^{\varepsilon h(x,t)}\Delta_x\Phi(x,z,t)\,\ d z.
\end{equation}
On the other hand, Leibniz' rule yields
\begin{equation}\label{eq:lub_Leibniz}
	\nabla_x\cdot\!\left(\int_{-1}^{\varepsilon h(x,t)}\nabla_x\Phi(x,z,t)\,dz\right)
	=\int_{-1}^{\varepsilon h(x,t)}\Delta_x\Phi(x,z,t)\,dz
	+\varepsilon\,\nabla_x h\cdot\nabla_x\Phi(x,\varepsilon h,t).
\end{equation}
Combining \eqref{eq:lub_dz_representation}--\eqref{eq:lub_Leibniz} and using $\sigma=\varepsilon\sqrt{\delta}$, we obtain
\[
\frac{1}{\sqrt{\delta}}\partial_z\Phi(x,\varepsilon h,t)
-\sigma\,\nabla_x h\cdot\nabla_x\Phi(x,\varepsilon h,t)
=-\sqrt{\delta}\,\nabla_x\cdot\!\left(\int_{-1}^{\varepsilon h(x,t)}\nabla_x\Phi(x,z,t)\,dz\right).
\]
Therefore, the kinematic condition in \eqref{eq:final_system_potential} can be written in conservative form as
\begin{equation}\label{eq:lub_evolution_flux}
	\partial_t h
	+\sqrt{\delta}\,\nabla_x\cdot\left(\int_{-1}^{\varepsilon h(x,t)}\nabla_x\Phi(x,z,t)\,dz\right)=0,
	\qquad x\in\TT^2.
\end{equation}

Together with the bulk equation and boundary conditions, we obtain the potential--flux formulation
\begin{equation}\label{eq:lub_potential_flux_system}
	\left\lbrace
	\begin{aligned}
		&(\delta\Delta_x+\partial_{zz})\Phi=0,
		&&\text{in }\Omega(t),\\[1mm]
		&\Phi=-\chi h-\lambda\,\mathcal E^\sigma(h)+\Theta\,\Delta_x\partial_t h,
		&&\text{on }\Gamma(t),\\[1mm]
		&\partial_z\Phi=0,
		&&\text{on }\Gamma_{\mathrm{bot}},\\[1mm]
		&\partial_t h
		+\sqrt{\delta}\,\nabla_x\cdot\left(\int_{-1}^{\varepsilon h(x,t)}\nabla_x\Phi(x,z,t)\,dz\right)=0,
		&& x\in\TT^2.
	\end{aligned}
	\right.
\end{equation}

\subsection{Flattening the moving strip}\label{subsec:lubrication_flattening}

To perform asymptotic expansions on a fixed geometry, we flatten the moving strip
\[
\Omega(t)=\{(x,z)\in\TT^2\times\RR:\,-1<z<\varepsilon h(x,t)\}
\]
onto the fixed strip
\[
S:=\TT^2\times(-1,0),\qquad \TT^2:=(-\pi,\pi)^2.
\]
We introduce a time-dependent diffeomorphism $\Sigma(t,\cdot):S\to\Omega(t)$ of the form
\begin{equation}\label{eq:lub_Sigma_def}
	\Sigma(t,x,z)=(x,\,z+\varphi(t,x,z)),
\end{equation}
where the lifting $\varphi$ satisfies
\begin{equation}\label{eq:lub_varphi_bc}
	\varphi(t,x,0)=\varepsilon h(x,t),\qquad \varphi(t,x,-1)=0,
\end{equation}
and we assume the non-degeneracy condition
\begin{equation}\label{eq:lub_diffeo_condition}
	1+\partial_z\varphi(t,x,z)\ge c_0>0.
\end{equation}
For lubrication computations we take the affine lifting
\begin{equation}\label{eq:lub_varphi_affine}
	\varphi(t,x,z)=\varepsilon(1+z)\,h(x,t).
\end{equation}

We define the pullback potential on $S$ by
\begin{equation}\label{eq:lub_phi_def}
	\phi(t,x,z):=\Phi\big(t,\Sigma(t,x,z)\big)=\Phi\big(t,x,z+\varphi(t,x,z)\big).
\end{equation}
Differentiating \eqref{eq:lub_phi_def} yields
\begin{equation}\label{eq:lub_chain_rule}
	\partial_{x_i}\phi=\partial_{x_i}\Phi+\partial_z\Phi\,\partial_{x_i}\varphi,
	\qquad
	\partial_z\phi=(1+\partial_z\varphi)\,\partial_z\Phi,
\end{equation}
and therefore
\begin{equation}\label{eq:lub_gradPhi_in_terms_phi}
	\partial_{x_i}\Phi=\partial_{x_i}\phi-\frac{\partial_{x_i}\varphi}{1+\partial_z\varphi}\,\partial_z\phi,
	\qquad
	\partial_z\Phi=\frac{1}{1+\partial_z\varphi}\,\partial_z\phi.
\end{equation}

In $\Omega(t)$, $\Phi$ satisfies $(\delta\Delta_x+\partial_{zz})\Phi=0$, which we write in divergence form as
\begin{equation}\label{eq:lub_div_form_bulk}
	\nabla_{x,z}\cdot\big(A_\delta\nabla_{x,z}\Phi\big)=0,
	\qquad
	A_\delta=\mathrm{diag}(\delta,\delta,1).
\end{equation}
A standard change-of-variables computation yields that $\phi$ solves, on the fixed strip,
\begin{equation}\label{eq:lub_elliptic_flat}
	\nabla_{x,z}\cdot\Big(P_\delta(\varphi)\,\nabla_{x,z}\phi\Big)=0
	\qquad\text{in }S,
\end{equation}
where
\begin{equation}\label{eq:lub_Pdelta_def}
	P_\delta(\varphi):=(\det D\Sigma)\,D\Sigma^{-1}\,A_\delta\,D\Sigma^{-T}.
\end{equation}
In the present graph-type flattening, one computes explicitly
\begin{equation}\label{eq:lub_Pdelta_explicit}
	P_\delta(\varphi)=
	\begin{pmatrix}
		\delta(1+\varphi_z) & 0 & -\delta\,\varphi_{x_1}\\
		0 & \delta(1+\varphi_z) & -\delta\,\varphi_{x_2}\\
		-\delta\,\varphi_{x_1} & -\delta\,\varphi_{x_2} &
		\dfrac{1+\delta|\nabla_x\varphi|^2}{1+\varphi_z}
	\end{pmatrix}.
\end{equation}
Since $z=0$ corresponds to the free surface $z=\varepsilon h(x,t)$ and $z=-1$ to the bottom, the boundary conditions become
\begin{equation}\label{eq:lub_BC_flat}
	\phi(\cdot,0,t)=-\chi h-\lambda\,\mathcal E^\sigma(h)+\Theta\,\Delta_x\partial_t h,
	\qquad
	\partial_z\phi(\cdot,-1,t)=0,
	\qquad x\in\TT^2.
\end{equation}

We now rewrite the kinematic condition \eqref{eq:lub_evolution_flux} on the fixed strip. Changing variables
$z'=z+\varphi(t,x,z)$ (so that $dz'=(1+\varphi_z)dz$ and $z'=\varepsilon h$ corresponds to $z=0$), and using
\eqref{eq:lub_gradPhi_in_terms_phi}, we obtain
\begin{equation}\label{eq:lub_flux_transform}
	\int_{-1}^{\varepsilon h(x,t)}\nabla_x\Phi(x,z',t)\,dz'
	=\int_{-1}^{0}\Big((1+\varphi_z)\,\nabla_x\phi-(\nabla_x\varphi)\,\partial_z\phi\Big)\,dz.
\end{equation}
Therefore the evolution law becomes
\begin{equation}\label{eq:lub_evolution_flat}
	\partial_t h
	+\sqrt{\delta}\,
	\nabla_x\cdot\left(\int_{-1}^{0}\Big((1+\varphi_z)\,\nabla_x\phi-(\nabla_x\varphi)\,\partial_z\phi\Big)\,dz\right)=0,
	\qquad x\in\TT^2.
\end{equation}
For the affine lifting \eqref{eq:lub_varphi_affine} one has $\varphi_z=\varepsilon h$ and
$\nabla_x\varphi=\varepsilon(1+z)\nabla_x h$, so that \eqref{eq:lub_evolution_flat} reduces to
\begin{equation}\label{eq:lub_evolution_flat_affine}
	\partial_t h+\sqrt{\delta}\,\nabla_x\cdot\left(
	(1+\varepsilon h)\int_{-1}^{0}\nabla_x\phi\,dz
	-\varepsilon\,\nabla_x h\int_{-1}^{0}(1+z)\,\partial_z\phi\,dz
	\right)=0,
	\qquad x\in\TT^2.
\end{equation}

% NOTE: You already said you do NOT want to duplicate the fixed-strip system.
% I recommend deleting the following summary block entirely, and only referencing
% \eqref{eq:lub_elliptic_flat}--\eqref{eq:lub_BC_flat} and \eqref{eq:lub_evolution_flat_affine}.
% (If you keep it, rename its label to avoid collisions with earlier sections.)

\subsection{A leading-order lubrication model}\label{subsec:lubrication_model}

Set
\begin{equation}\label{eq:lub_f_def}
	f(x,t):=\phi(x,0,t)
	=-\chi h-\lambda\,\mathcal E^\sigma(h)+\Theta\,\Delta_x\partial_t h.
\end{equation}
To extract the leading long-wave behaviour as $\delta\to0$, we expand the pullback potential as
\begin{equation}\label{eq:lub_phi_expansion}
	\phi=\phi^{0}+\mathcal{O}(\delta)
	\qquad \text{in } S,
\end{equation}
with boundary conditions $\phi^{0}|_{z=0}=f$, $\partial_z\phi^{0}|_{z=-1}=0$.
At leading order one finds $\partial_{zz}\phi^{0}=0$, hence
\begin{equation}\label{eq:lub_phi0}
	\phi^{0}(x,z,t)=f(x,t).
\end{equation}
Using \eqref{eq:lub_phi_expansion} we compute the fluxes in \eqref{eq:lub_evolution_flat_affine}. First,
\begin{equation}\label{eq:lub_int_gradphi}
	\int_{-1}^{0}\nabla_x\phi\,dz
	=\nabla_x f+\mathcal{O}(\delta),
\end{equation}
and, since $\partial_z\phi^{0}\equiv0$,
\begin{equation}\label{eq:lub_int_zphiz}
	\int_{-1}^{0}(1+z)\,\partial_z\phi\,dz
	=\mathcal{O}(\delta).
\end{equation}
Substituting \eqref{eq:lub_int_gradphi}--\eqref{eq:lub_int_zphiz} into \eqref{eq:lub_evolution_flat_affine} yields
\begin{equation}\label{eq:lub_balance_leading}
	\partial_t h+\sqrt{\delta}\,\nabla_x\cdot\Big((1+\varepsilon h)\,\nabla_x f\Big)
	=\mathcal{O}(\delta^{3/2}),
	\qquad f \text{ given by \eqref{eq:lub_f_def}}.
\end{equation}
Truncating at the first nontrivial order gives the leading lubrication model
\begin{equation}\label{eq:lub_model_closed}
	\partial_t h+\sqrt{\delta}\,\nabla_x\cdot\Big((1+\varepsilon h)\,\nabla_x f\Big)=0,
	\qquad
	f=-\chi h-\lambda\,\mathcal E^\sigma(h)+\Theta\,\Delta_x\partial_t h.
\end{equation}

\smallskip

In the long-wave scaling, the physical slope satisfies $|\nabla_x h_{\rm phys}|\sim H/L=\varepsilon\sqrt{\delta}$, hence the
first nonlinear corrections in the elastic operator are of size $\varepsilon^{2}\delta$. Consequently, at the accuracy of
\eqref{eq:lub_model_closed} it is consistent to use the linear approximation
\begin{equation}\label{eq:lub_elastic_linearization}
	\mathcal E^\sigma(h)=\frac14\,\Delta_x^{2}h+\mathcal{O}(\varepsilon^{2}\delta).
\end{equation}
Substituting \eqref{eq:lub_elastic_linearization} into \eqref{eq:lub_model_closed} and discarding the remainder
$\mathcal{O}(\varepsilon^{2}\delta^{3/2})$ yields the explicit thin-film equation
\begin{equation}\label{eq:lub_PDE_explicit}
	\partial_t h+\sqrt{\delta}\,\nabla_x\cdot\Big((1+\varepsilon h)\,\nabla_x\mu\Big)=0,
	\qquad
	\mu:=\Theta\,\Delta_x\partial_t h-\chi h-\frac{\lambda}{4}\,\Delta_x^{2}h.
\end{equation}
Expanding the mobility $(1+\varepsilon h)$ separates the linear and nonlinear contributions:
\begin{equation}\label{eq:lub_PDE_explicit_split}
	\partial_t h+\sqrt{\delta}\,\Delta_x\mu=-\sqrt{\delta}\,\varepsilon\,\nabla_x\cdot\big(h\,\nabla_x\mu\big),
	\qquad \mu \text{ as in \eqref{eq:lub_PDE_explicit}}.
\end{equation}

\begin{remark}[One-dimensional analogue]\label{rem:lub_1D}
	For a one-dimensional interface $h=h(x,t)$, $x\in\TT$, the explicit leading-order thin-film model
	\eqref{eq:lub_PDE_explicit} reads
	\begin{equation}\label{eq:lub_PDE_explicit_1D}
		\partial_t h+\sqrt{\delta}\,\partial_x\Big((1+\varepsilon h)\,\partial_x\mu\Big)=0,
		\qquad
		\mu:=\Theta\,\partial_{xx}\partial_t h-\chi h-\frac{\lambda}{4}\,\partial_x^{4}h.
	\end{equation}
	Equivalently,
	\begin{equation}\label{eq:lub_PDE_explicit_1D_split}
		\partial_t h+\sqrt{\delta}\,\partial_x^{2}\mu
		=-\sqrt{\delta}\,\varepsilon\,\partial_x\big(h\,\partial_x\mu\big),
		\qquad \mu \text{ as in \eqref{eq:lub_PDE_explicit_1D}}.
	\end{equation}
	In particular, the left-hand side is the linear part, while the right-hand side collects the quadratic transport
	coming from the variable mobility.
\end{remark}

\section{Notation, functional spaces, and basic inequalities}
\label{sec:notation_spaces}

We work on the one-dimensional torus $\TT=\RR/(2\pi\ZZ)$. For $f:\TT\to\CC$ we use the Fourier series convention
\[
f(x)=\sum_{k\in\ZZ}\widehat f(k)\,\frac{e^{ikx}}{\sqrt{2\pi}},
\qquad
\widehat f(k)=\int_{\TT} f(x)\,\frac{e^{-ikx}}{\sqrt{2\pi}}\,dx.
\]
The spatial mean of $f$ is $\widehat f(0)$. Unless explicitly stated otherwise, all functions considered in this paper are assumed to have zero spatial mean, i.e.\ $\widehat f(0)=0$ (this condition is preserved by the evolutions studied here). Throughout the paper, $C>0$ denotes a generic constant whose value may change from line to line.
When relevant, we indicate the dependence on parameters, e.g.\ $C=C(\chi,\Theta,\sigma,\lambda)$. \medskip

\subsubsection*{Functional spaces and inequalities} 
For $s\ge 0$ we define the periodic Wiener space for zero mean functions. To streamline notation, we do not introduce separate symbols for mean-zero subspaces.
Hence, whenever we write $f\in A^s(\TT)$ we implicitly assume $\widehat f(0)=0$. Therefore,
\[
A^s(\TT):=\Bigl\{f\in \mathcal D'(\TT):\ \|f\|_{A^s}<\infty\Bigr\},
\qquad
\|f\|_{A^s}:=\sum_{k\in\ZZ}|k|^s\,|\widehat f(k)|.
\]
If $s\in\NN$, this agrees with the characterization
\[
A^s(\TT)=\Bigl\{f:\ \|\widehat{\partial_x^s f}\|_{\ell^1}<\infty\Bigr\}.
\]

We will repeatedly use the following well-known standard inequalities, \cite{BG-B2019}.

\medskip
\noindent\emph{(i) \textcolor{black}{Product} estimates.}
The \textcolor{black}{norm of the} space $A^0(\TT)$ \textcolor{black}{verifies the following product inequality}:
\[
\|fg\|_{A^0}\le \|f\|_{A^0}\,\|g\|_{A^0}.
\]
More generally, for $s\ge 0$ one has the Leibniz-type bound
\[
\|fg\|_{A^s}\le C_s\bigl(\|f\|_{A^s}\|g\|_{A^0}+\|f\|_{A^0}\|g\|_{A^s}\bigr).
\]
In particular, if $s\ge 1$ and $f,g\in A^s\textcolor{black}{(\TT)}$, then $fg\in A^s\textcolor{black}{(\TT)}$ and
$\|fg\|_{A^s}\le C_s\|f\|_{A^s}\|g\|_{A^s}$.

\medskip
\noindent\emph{(ii) Poincar\'e-type inequalities for mean-zero functions.}
Under the standing assumption $\widehat f(0)=0$, for $0\le m\le n$,
\[
\|f\|_{A^m}\le \|f\|_{A^n},
\qquad\text{and in particular}\qquad
\|f\|_{A^0}\le \|f\|_{A^1}.
\]

\medskip
\noindent\emph{(iii) Derivatives.}
For integers $j\ge 0$ and \textcolor{black}{zero-mean functions, we have that}
\[
\|\partial_x^j f\|_{A^s}=\|f\|_{A^{s+j}}.
\]

\medskip
\noindent\emph{(iv) Fourier multipliers.}
If $M$ is a Fourier multiplier with symbol $m(k)$ such that
$\sup_{k\in\ZZ\setminus\{0\}} |m(k)|\,|k|^{-\alpha}\le C$ for some $\alpha\in\RR$, then
\[
\|Mf\|_{A^s}\le C\,\|f\|_{A^{s+\alpha}}.
\]
In particular, for
\[
\textcolor{black}{\mathcal G_0}:=\Lambda\tanh(\Lambda),\qquad \widehat{\textcolor{black}{\mathcal G_0} f}(k)=|k|\tanh(|k|)\,\widehat f(k),
\]
we use that 
\textcolor{black}{$$
\tanh(1)\leq\tanh(|k|)\le 1, \qquad k\in\ZZ\setminus\{0\},
$$} to obtain, for all $s\ge 0$ and a zero mean function $f$,
\[
\|\textcolor{black}{\mathcal G_0}f\|_{A^s}\le \|f\|_{A^{s+1}},
\qquad
\|\textcolor{black}{\mathcal G_0} f\|_{A^s}\ge \tanh(1)\,\|f\|_{A^{s+1}}.
\]

\subsubsection*{Estimates on the bilinear mappings} \textcolor{black}{ We first record the bilinear estimates that will be used throughout the well-posedness analysis.}
\textcolor{black}{\begin{lemma}\label{LemmaBilinear}
We have the estimates
$$
\|\textcolor{black}{\mathcal G_0}\big(h\,\textcolor{black}{\mathcal G_0}\partial_{xx}V\big)+\partial_x\big(h\,\partial_{xxx}V\big)\|_{A^0}\leq C\,\|h\|_{A^1}\,\|V\|_{A^3},
$$
$$
\|\textcolor{black}{\mathcal G_0}\big(h\,\textcolor{black}{\mathcal G_0}V\big)+\partial_x\big(h\,\partial_{x}V\big)\|_{A^0}\leq C\,\|h\|_{A^1}\,\|V\|_{A^1}.
$$
$$
\|\textcolor{black}{\mathcal G_0}\big(h\,\textcolor{black}{\mathcal G_0}\partial_{xxxx}V\big)+\partial_x\big(h\,\partial_{xxxxx}V\big)\|_{A^0}\leq C\,\|h\|_{A^1}\,\|V\|_{A^5},
$$
\end{lemma} 
\begin{proof}
We are going to focus on the proof of the first estimate. Once this is achieved, the other two can be recovered from it by choosing $V$ appropriately.
		Set
	\[
	I(h,V):=\textcolor{black}{\mathcal G_0}\big(h\,\textcolor{black}{\mathcal G_0}\partial_{xx}V\big)+\partial_x\big(h\,\partial_{xxx}V\big).
	\]
	We now compute $\widehat I(k)$ and bound $\|I(h,V)\|_{A^0}$ in detail.
First, we notice that
	\begin{align}
		\widehat I(k)
		&=
		|k|\tanh(|k|)\sum_{m\in\ZZ}\widehat h(m)\,\widehat{\textcolor{black}{\mathcal G_0}\partial_{xx}V}(k-m)
		+\sum_{m\in\ZZ}k(k-m)^3\widehat h(m)\widehat V(k-m)\notag\\
		&=\sum_{m\in\ZZ}\Big[-|k||k-m|^3\tanh(|k|)\tanh(|k-m|)+k(k-m)^3\Big]\widehat h(m)\widehat V(k-m).
		\label{eq:Ihat_integrated_raw}
	\end{align}
	Rewrite $k(k-m)^3=|k||k-m|^3\,\sgn(k)\sgn(k-m)$ to obtain
	\begin{equation}\label{eq:Ihat_integrated_factor}
		\widehat I(k)=\sum_{m\in\ZZ}|k||k-m|^3
		\Big[\sgn(k)\sgn(k-m)-\tanh(|k|)\tanh(|k-m|)\Big]\widehat h(m)\widehat V(k-m).
	\end{equation}
We decompose the symbol bracket into two contributions, isolating the discontinuous
(sign) interaction from the smooth remainder:
\[
\sgn(k)\sgn(k-m)-\tanh(|k|)\tanh(|k-m|)
=
\underbrace{\big(\sgn(k)\sgn(k-m)-1\big)}_{=:A_{k,m}}
+
\underbrace{\big(1-\tanh(|k|)\tanh(|k-m|)\big)}_{=:B_{k,m}}.
\]
Accordingly, we write $\widehat I=\widehat I_A+\widehat I_B$, where $\widehat I_A$ (resp.\ $\widehat I_B$)
is obtained from \eqref{eq:Ihat_integrated_factor} by replacing the bracket with $A_{k,m}$ (resp.\ $B_{k,m}$).
We estimate these two terms separately. We note that $A_{k,m}=0$ unless $\sgn(k)\neq\sgn(k-m)$, in which case $|A_{k,m}|\le 2$ and necessarily $|k|\le |m|$.
Therefore,
\[
|\widehat I_A(k)|
\le 2\sum_{m\in\ZZ}|k||k-m|^3\,\mathbf 1_{\{|k|\le |m|\}}\,|\widehat h(m)|\,|\widehat V(k-m)|
\le 2\sum_{m\in\ZZ}|m|\,|k-m|^3\,|\widehat h(m)|\,|\widehat V(k-m)|.
\]
Summing over $k$ and changing variables $j=k-m$ yields
\[
\|I_A\|_{A^0}
\le 2\sum_{m\in\ZZ}|m||\widehat h(m)|\sum_{j\in\ZZ}|j|^3|\widehat V(j)|
=2\,\|h\|_{A^1}\,\|V\|_{A^3}.
\]
Next, let us estimate $I_{B}$. To that purpose, for $a,b\ge 0$ we use the elementary bound
\[
1-\tanh a\,\tanh b=(1-\tanh a)+\tanh a(1-\tanh b)\le (1-\tanh a)+(1-\tanh b),
\]
together with $1-\tanh r=\frac{2}{e^{2r}+1}\le 2e^{-2r}$, to obtain
\[
|B_{k,m}|\le C\big(e^{-2|k|}+e^{-2|k-m|}\big).
\]
Inserting this bound into \eqref{eq:Ihat_integrated_factor} gives
\[
|\widehat I_B(k)|
\le C\sum_{m\in\ZZ}|k||k-m|^3\big(e^{-2|k|}+e^{-2|k-m|}\big)\,|\widehat h(m)|\,|\widehat V(k-m)|
=:J_1(k)+J_2(k),
\]
where $J_1$ corresponds to the factor $e^{-2|k|}$ and $J_2$ to $e^{-2|k-m|}$.
For $J_1$, since $\sum_{k\in\ZZ}|k|e^{-2|k|}<\infty$,
\[
\sum_{k\in\ZZ}J_1(k)
\le C\Big(\sum_{k\in\ZZ}|k|e^{-2|k|}\Big)\sum_{m\in\ZZ}|\widehat h(m)|\sum_{j\in\ZZ}|j|^3|\widehat V(j)|
\le C\,\|h\|_{A^0}\,\|V\|_{A^3}.
\]
For $J_2$, set $j=k-m$, so that $k=m+j$ and $|k|\le |m|+|j|$:
\begin{align*}
	\sum_{k\in\ZZ}J_2(k)
	&\le C\sum_{m\in\ZZ}\sum_{j\in\ZZ}(|m|+|j|)\,|j|^3e^{-2|j|}\,|\widehat h(m)|\,|\widehat V(j)|\\
	&\le C\|h\|_{A^1}\|V\|_{A^3}
	+ C\|h\|_{A^0}\sum_{j\in\ZZ}|j|^4e^{-2|j|}|\widehat V(j)|.
\end{align*}
Since $|j|e^{-2|j|}\le C$, we have
$\sum_{j\in\ZZ}|j|^4e^{-2|j|}|\widehat V(j)|
\le C\sum_{j\in\ZZ}|j|^3|\widehat V(j)|
= C\|V\|_{A^3}$,
and thus
\[
\|I_B\|_{A^0}\le C\big(\|h\|_{A^1}+\|h\|_{A^0}\big)\|V\|_{A^3}.
\]
	Hence
	\[
	\|I_B\|_{A^0}\le C(\|h\|_{A^1}+\|h\|_{A^0})\|V\|_{A^3}.
	\]
	Using Poincar\'e on mean-zero functions, $\|h\|_{A^0}\le \|h\|_{A^1}$, we conclude
	\begin{equation}\label{eq:I_bound_integrated}
		\|I(h,V)\|_{A^0}\le C\,\|h\|_{A^1}\,\|V\|_{A^3}.
	\end{equation}
	\textcolor{black}{
Since $\|h\|_{A^0}\leq\|h\|_{A^1}$ for the full norms defined above, we conclude that
	\begin{equation}\label{eq:I_bound_integrated}
		\|I(h,V)\|_{A^0}\le C\,\|h\|_{A^1}\,\|V\|_{A^3}.
	\end{equation}
To obtain the second inequality, first remove the irrelevant constant mode
of $v$ and apply the first one with $V=\partial_{xx}^{-1}v$, then
$\|\partial_{xx}^{-1}v\|_{A^3}\leq\|v\|_{A^1}$. To obtain the third
inequality, apply the first one with $V=\partial_{xx}v$ and use
$\|\partial_{xx}v\|_{A^3}\leq\|v\|_{A^5}$.}
\end{proof}
}

\section{Well-posedness of the weakly nonlinear models for the small slope regime}\label{sec:wp1}
We now study the well-posedness of the weakly nonlinear models introduced in the previous section. To streamline the presentation, we restrict to the one-dimensional setting. \textcolor{black}{Due to the change in the order of the PDE, the cases $\lambda=0$ and $\lambda\neq0$ are very different. In fact, the case $\lambda=0$ could be regarded intuitively as an ODE in a Banach space while the case $\lambda\neq0$ cannot. Thus, we treat both cases in separate results with different statements. Moreover when $\lambda>0$, the elastic
term supplies fifth-order spatial dissipation before inversion of the operator
acting on $\partial_t h$, and this yields the additional space-time control.}\medskip 

Our main theorem is as follows.

\begin{theorem}\label{thm:1}
	\label{thm:wp_A1}
	Let $\lambda=0$ and fix $\chi,\Theta,\sigma>0$. Assume that $h_0\in A^1\textcolor{black}{(\TT)}$ has zero mean.
	Then there exist constants $\mathcal C=\mathcal C(\chi,\Theta,\sigma)>0$ and
	$T=T(\chi,\Theta,\sigma,\|h_0\|_{A^1})>0$ such that, if
	\[
	\|h_0\|_{A^1}\le \mathcal C,
	\]
	there exists a (mild) solution
	\[
	h\in C([0,T],A^1\textcolor{black}{(\TT)})
	\]
	to \eqref{eq:model_sigma_LHS_full_multiplier_1D} satisfying the integral formulation
	\[
	h(t)=h_0+\int_0^t (\mathcal L_{h(s)})^{-1}\,\mathcal N(h(s))\,ds,
	\qquad t\in[0,T],
	\]
	where
	\[
	\mathcal N(h)
	=
	-\chi\,\Lambda\tanh(\Lambda)\,h
	+\sigma\chi\Big(\Lambda\tanh(\Lambda)\big(h\,\Lambda\tanh(\Lambda)\,h\big)
	+\partial_x\big(h\,\partial_x h\big)\Big),
	\]
	and, for a given profile $h$, the operator $\mathcal L_h$ is defined by
	\[
	\mathcal{L}_h U
	=
	\big(1-\Theta\,\Lambda\tanh(\Lambda)\,\partial_{xx}\big)U
	+\sigma\Theta\Big(\Lambda\tanh(\Lambda)\big(h\,\Lambda\tanh(\Lambda)\,\partial_{xx}U\big)
	+\partial_x\big(h\,\partial_{xxx}U\big)\Big).
	\]
		Moreover, \textcolor{black}{there exists another constant $\mathcal C_2=\mathcal C_2(\lambda,\chi,\Theta,\sigma)$ such that if $h_0\in A^3\textcolor{black}{(\TT)}$
with
$$
\|h_0\|_{A^0}+\Theta\tanh(1)\|h_0\|_{A^3}\leq \mathcal{C}_2,
$$
then the corresponding solution satisfies
	\[
	h\in C([0,\infty),A^3\textcolor{black}{(\TT)}),
	\]
	and the decay property
	\[
	\limsup_{t\to\infty}\|h(t)\|_{A^3}=0.
	\]
}
\end{theorem}
\textcolor{black}{
\begin{remark}\label{rem:Lh_invertible_minimal}
Lemma~\ref{LemmaBilinear} and \eqref{eq:L0_inv_integrated} show that, if
$\|h\|_{A^1}$ is sufficiently small, then
$\mathcal L_h:A^3(\TT)\to A^0(\TT)$ is boundedly invertible on zero-mean
functions. Indeed,  we show that if 
$\mathcal L_h$ is a small perturbation of $\mathcal L_0$, we can apply a Banach fixed point argument with
\[
\|\mathcal L_h^{-1}F\|_{A^0}
+\Theta\tanh(1)\|\mathcal L_h^{-1}F\|_{A^3}
\le C\|F\|_{A^0}.
\]
Thus the integral formulation in Theorem~\ref{thm:wp_A1} is well defined.
\end{remark}
}
\begin{proof}[Proof of Theorem \ref{thm:1}]
	Set $\textcolor{black}{\mathcal G_0}:=\Lambda\tanh(\Lambda)$, i.e.\ the Fourier multiplier with symbol $|k|\tanh(|k|)$, and define
\[
\mathcal L_0 := 1-\Theta\,\textcolor{black}{\mathcal G_0}\partial_{xx}.
\]
We work on $\TT$ and assume that the initial datum has zero mean, $\widehat h_0(0)=0$.
This property is preserved by the evolution, since every term on the right-hand side of
\eqref{eq:model_sigma_LHS_full_multiplier_1D} has zero spatial mean.
For $T>0$ and 
\textcolor{black}{
$$
\|h_0\|_{A^1}\le \mathcal C
$$
with
}
$\mathcal C>0$ define
	\[
	\mathbb X_T^{\mathcal C}
	:=
	\Big\{h\in C([0,T];A^0)\cap L^\infty(0,T;A^1):\ h(0)=h_0,\ 
	\|h\|_{L^\infty(0,T;A^1)}\le \textcolor{black}{2}\mathcal C\Big\}.
	\]
	\textcolor{black}{Thus the initial datum lies in the half-radius ball,
	leaving the margin needed in the self-map argument.}
	With $\lambda=0$, equation \eqref{eq:model_sigma_LHS_full_multiplier_1D} can be written as
	\begin{equation}\label{eq:ql_form_integrated}
		\mathcal L_h(\partial_t h)=\mathcal N(h),
	\end{equation}
	where
	\[
	\mathcal N(h):=-\chi\,\textcolor{black}{\mathcal G_0}h+\sigma\chi\Big(\textcolor{black}{\mathcal G_0}(h\,\textcolor{black}{\mathcal G_0}h)+\partial_x(h\,\partial_x h)\Big),
	\]
	and
	\[
	\mathcal L_h U
	:=
	\mathcal L_0U
	+\sigma\Theta\Big(\textcolor{black}{\mathcal G_0}\big(h\,\textcolor{black}{\mathcal G_0}\partial_{xx}U\big)+\partial_x\big(h\,\partial_{xxx}U\big)\Big).
	\]
	
	\medskip
	\underline{Step 1:  Invertibility of $\mathcal L_0$.}
	\textcolor{black}{Let $F\in A^0(\TT)$ be a function with zero integral mean} and set $U=\mathcal L_0^{-1}F$. For $k\neq 0$,
	\[
	\widehat{(\mathcal L_0U)}(k)=\ell_0(k)\,\widehat U(k),
	\qquad
	\ell_0(k)=1+\Theta |k|^3\tanh(|k|),
	\]
	hence $\widehat U(k)=\widehat F(k)/\ell_0(k)$. Since $\ell_0(k)\ge 1$,
	\[
	\|U\|_{A^0}\le \|F\|_{A^0}.
	\]
	Moreover, using $\tanh(|k|)\ge \tanh(1)$ for $k\neq 0$,
	\[
	\Theta\tanh(1)\,|k|^3|\widehat U(k)|
	\le \frac{\Theta\tanh(1)\,|k|^3}{1+\Theta|k|^3\tanh(|k|)}\,|\widehat F(k)|
	\le |\widehat F(k)|,
	\]
	and summing in $k\neq 0$ yields
	\begin{equation}\label{eq:L0_inv_integrated}
		\|\mathcal L_0^{-1}F\|_{A^0}+\Theta\tanh(1)\,\|\mathcal L_0^{-1}F\|_{A^3}
		\le C\,\|F\|_{A^0}.
	\end{equation}
	
	\medskip
	\underline{Step 2: first fixed point: solve $\mathcal L_h U=\mathcal N(h)$ for fixed $h$.}
	Fix $h\in\mathbb X_T^{\mathcal C}$. For $V\in L^\infty(0,T;A^3)$ define $S_h(V)=U$ by
	\begin{equation}\label{eq:Sh_integrated}
		\mathcal L_0U
		=
		\mathcal N(h)-\sigma\Theta\Big(\textcolor{black}{\mathcal G_0}\big(h\,\textcolor{black}{\mathcal G_0}\partial_{xx}V\big)+\partial_x\big(h\,\partial_{xxx}V\big)\Big).
	\end{equation}

Invoking Lemma \ref{LemmaBilinear}, we find
$$
\|\textcolor{black}{\mathcal G_0}\big(h\,\textcolor{black}{\mathcal G_0}\partial_{xx}V\big)+\partial_x\big(h\,\partial_{xxx}V\big)\|_{A^0}\leq C\,\|h\|_{A^1}\,\|V\|_{A^3}.
$$
	
Now apply \eqref{eq:L0_inv_integrated} to \eqref{eq:Sh_integrated}.
	The second estimate in Lemma~\ref{LemmaBilinear} gives
	$\|\mathcal N(h)\|_{A^0}\le C(\|h\|_{A^1}+\|h\|_{A^1}^2)$, which together with \eqref{eq:I_bound_integrated} shows that
	\begin{align}
		\|S_h(V)\|_{A^0}+\Theta\tanh(1)\|S_h(V)\|_{A^3}
		&\le C\Big(\|\mathcal N(h)\|_{A^0}+\sigma\Theta\|I(h,V)\|_{A^0}\Big)\notag\\
		&\le C\Big(\|h\|_{A^1}+\|h\|_{A^1}^2+\|h\|_{A^1}\|V\|_{A^3}\Big),
		\label{eq:Sh_est_integrated}
	\end{align}
  for each $t\in[0,T]$,
	Similarly, by linearity in $V$ of the right-hand side of \eqref{eq:Sh_integrated} and \eqref{eq:I_bound_integrated},
	\[
	\|S_h(V_1)-S_h(V_2)\|_{A^3}
	\le C\,\|h\|_{A^1}\,\|V_1-V_2\|_{A^3}.
	\]
	Therefore, if $\mathcal C>0$ is chosen so that $C\mathcal C<1$, then $S_h$ is a contraction on $L^\infty(0,T;A^3)$.
	Hence there exists a unique fixed point $U_h\in L^\infty(0,T;A^3)$ such that
	\begin{equation}\label{eq:Uh_fixed_integrated}
		U_h=S_h(U_h)\qquad\Longleftrightarrow\qquad \mathcal L_hU_h=\mathcal N(h).
	\end{equation}
	Moreover, taking $\mathcal C$ smaller if needed, \eqref{eq:Sh_est_integrated} at $V=U_h$ yields
	\begin{equation}\label{eq:Uh_bound_integrated}
		\|U_h\|_{A^0}+\frac{\Theta\tanh(1)}{2}\|U_h\|_{A^3}
		\le C\,\|h\|_{A^1}\big(1+\|h\|_{A^1}\big).
	\end{equation}
	
\underline{Step 3: the second fixed point: construction of $h$.}
Define $\mathcal T$ on $\mathbb X_T^{\mathcal C}$ by
\[
(\mathcal Th)(t):=
h_0+\int_0^tU_h(s)\,ds,
\qquad t\in[0,T].
\]
Since $U_h\in L^\infty(0,T;A^3)$, the map $\mathcal Th$ belongs to
$C([0,T];A^1)$.
	By \eqref{eq:Uh_bound_integrated},
	\[
	\|\mathcal Th\|_{L^\infty(0,T;A^1)}
	\le \|h_0\|_{A^1}
	+TC\,\|h\|_{L^\infty(0,T;A^1)}\big(1+\|h\|_{L^\infty(0,T;A^1)}\big).
	\]
	Hence, choosing $T>0$ sufficiently small (depending on $\|h_0\|_{A^1}$ and parameters) and $\mathcal C$ small,
	we have $\mathcal T(\mathbb X_T^{\mathcal C})\subset \mathbb X_T^{\mathcal C}$.	For the contraction arguement, take $h_1,h_2\in\mathbb X_T^{\mathcal C}$ and denote $\delta h=h_1-h_2$,
	$\delta U=U_{h_1}-U_{h_2}$. Subtracting the identities $\mathcal L_{h_j}U_{h_j}=\mathcal N(h_j)$,
	one obtains an equation of the form
	\[
	\mathcal L_0(\delta U)=F(\delta h)+G(\delta h,U_{h_1})+H(h_2,\delta U),
	\]
	where each term is linear in $\delta h$ or $\delta U$.
	Using again \eqref{eq:L0_inv_integrated}, the bound \eqref{eq:I_bound_integrated}, and Wiener algebra estimates,
	we obtain
	\[
	\|\delta U\|_{A^0}+\frac{\Theta\tanh(1)}{2}\|\delta U\|_{A^3}
	\le C\,\|\delta h\|_{A^1},
	\]
	with $C$ depending on $\mathcal C$ (hence on the fixed parameters) but not on $T$.
	Therefore,
	\[
	\|\mathcal Th_1-\mathcal Th_2\|_{L^\infty(0,T;A^1)}
	\le T\,C\,\|h_1-h_2\|_{L^\infty(0,T;A^1)}.
	\]
	Taking $T$ smaller so that $TC<1$, Banach's fixed point theorem gives a unique fixed point
	$h\in\mathbb X_T^{\mathcal C}$ with $\mathcal Th=h$, i.e.\ a mild solution on $[0,T]$.
	
	\medskip
	\underline{Step 4: Additional $A^3$ regularity, global existence, and decay.}
	\textcolor{black}{For the local regularity statement, assume now that
	$h_0\in A^3$ and $\|h_0\|_{A^1}\leq\mathcal C$, no smallness of the
	$A^3$ norm is used at this stage.}
	From \eqref{eq:Uh_fixed_integrated} we have $\partial_t h=U_h\in L^\infty(0,T;A^3)$, hence
	\[
	h(t)=h_0+\int_0^t \partial_s h(s)\,ds \in A^3
	\quad\text{for all }t\in[0,T],
	\]
	and thus $h\in C([0,T];A^3)$. \medskip
	
	We next derive an a priori estimate that extends the solution globally.
	Taking Fourier coefficients in \eqref{eq:ql_form_integrated} and using
	$\frac{d}{dt}|\widehat h(k)|\le |\widehat{\partial_t h}(k)|$, we obtain
	\begin{align}
		\frac{d}{dt}\|h\|_{A^0}
		&\le \|\partial_t h\|_{A^0}.
		\label{eq:A0_time_derivative}
	\end{align}
	Now, from \eqref{eq:ql_form_integrated} we can write
	\[
	\mathcal L_0(\partial_t h)=\mathcal N(h)-\sigma\Theta I(h,\partial_t h).
	\]
	Applying \eqref{eq:L0_inv_integrated} with $F=\mathcal N(h)-\sigma\Theta I(h,\partial_t h)$, and using
	\eqref{eq:I_bound_integrated} with $V=\partial_t h$, yields
	\begin{align}
		\|\partial_t h\|_{A^0}+\Theta\tanh(1)\|\partial_t h\|_{A^3}
		&\le C\Big(\|\mathcal N(h)\|_{A^0}+\sigma\Theta\|I(h,\partial_t h)\|_{A^0}\Big)\notag\\
		&\le C\Big(\|h\|_{A^1}+\|h\|_{A^1}^2+\|h\|_{A^1}\|\partial_t h\|_{A^3}\Big).
		\label{eq:dt_h_est_closure}
	\end{align}
	Choose $\mathcal C>0$ so that $C\mathcal C\le \frac12\,\Theta\tanh(1)$ (shrinking $\mathcal C$ if necessary).
	Then, since $\|h\|_{A^1}\le \mathcal C$, the last term in \eqref{eq:dt_h_est_closure} can be absorbed, giving
	\begin{equation}\label{eq:dt_h_closed}
		\|\partial_t h\|_{A^0}+\frac{\Theta\tanh(1)}{2}\|\partial_t h\|_{A^3}
		\le C\big(\|h\|_{A^1}+\|h\|_{A^1}^2\big).
	\end{equation}
\textcolor{black}{Multiplying the $k$th Fourier equation by $
\frac{\overline{\widehat h(k)}}{|\widehat h(k)|},$
taking real parts, and summing over $k\neq0$, we obtain the derivative
\[
\frac{d}{dt}
\left(
\|h\|_{A^0}+\Theta\|\tanh(\Lambda)\Lambda^{3}h\|_{A^0}
\right).
\]
Therefore, applying the second and first estimates of
Lemma~\ref{LemmaBilinear} with $V=h$ and $V=\partial_t h$, respectively,
gives
}
	\begin{equation}\label{eq:energy_like_decay}
		\frac{d}{dt}\Big(\|h\|_{A^0}+\Theta\tanh(1)\|h\|_{A^3}\Big)+\chi\textcolor{black}{\tanh(1)}\,\|h\|_{A^1}
		\le C\,\|h\|_{A^1}^2 + C\,\|h\|_{A^1}\,\|\partial_t h\|_{A^3}.
	\end{equation}
	Using \eqref{eq:dt_h_closed} to bound $\|\partial_t h\|_{A^3}$ in terms of $\|h\|_{A^1}$ and the smallness
	$\|h\|_{A^1}\le \mathcal C$, we can absorb the right-hand side and obtain
	\begin{equation}\label{eq:energy_decay_absorbed}
		\frac{d}{dt}\Big(\|h\|_{A^0}+\Theta\tanh(1)\|h\|_{A^3}\Big)+\frac{\chi}{2}\textcolor{black}{\tanh(1)}\,\|h\|_{A^1}\le 0.
	\end{equation}
	In particular, $t\mapsto \|h(t)\|_{A^0}+\Theta\tanh(1)\|h(t)\|_{A^3}$ is decreasing for non-zero functions.
\textcolor{black}{
	Integrating \eqref{eq:energy_decay_absorbed}, we obtain
\begin{equation}\label{eq:energy_decay_absorbed2}
\begin{aligned}
&\|h(t)\|_{A^0}+\Theta\tanh(1)\|h(t)\|_{A^3}
+\frac{\chi\tanh(1)}{2}
\int_0^t\|h(s)\|_{A^1}\,ds \leq
\|h_0\|_{A^0}+\Theta\tanh(1)\|h_0\|_{A^3}.
\end{aligned}
\end{equation}
Since $h$ has zero mean, using interpolation we find that
\[
\min\{1,\Theta\tanh(1)\}\|h(t)\|_{A^1}
\leq
\|h(t)\|_{A^0}+\Theta\tanh(1)\|h(t)\|_{A^3}.
\]
Consequently,
\[
\sup_{t\geq0}\|h(t)\|_{A^1}
\leq
\frac{\mathcal C_2}{\min\{1,\Theta\tanh(1)\}}.
\]
Choosing $\mathcal C_2$ sufficiently small gives a strict improvement of
the $A^1$ bootstrap, and the local solution therefore extends globally.} \medskip

\textcolor{black}{
It remains to justify the asserted $A^3$ decay of the solution. From the energy balance inequality \eqref{eq:energy_decay_absorbed2}, we have that the $\|h(t)\|_{A^0}+\Theta\tanh(1)\|h(t)\|_{A^3}$ converges monotonically to zero. Therefore, we can infer that each term converges (maybe not monotonically) to zero, yielding the desired decay. }

\end{proof}

Let us now turn our attention to the case $\lambda>0.$ For this case we have the following theorem:

\begin{theorem}\label{thm:wp_A3_lambda}
	Fix $\lambda,\chi,\Theta,\sigma>0$ and let $h_0\in A^3(\TT)$ have zero mean.
	There exist two constants $\mathcal C=\mathcal C(\lambda,\chi,\Theta,\sigma),\textcolor{black}{\mathcal C_2=\mathcal C_2(\lambda,\chi,\Theta,\sigma)}>0$ such that, if
	\[
	\|h_0\|_{A^1}\le \mathcal C,
	\]
	and

	$$
	\textcolor{black}{\left(\|h_0\|_{A^0}+\Theta\tanh(1)\|h_0\|_{A^3}\right)}\leq \mathcal{C}_2,
	$$
	equation \eqref{eq:model_sigma_LHS_full_multiplier_1D} admits a (mild) solution
	\[
	h\in C([0,\infty);A^3(\TT)).
	\]
	More precisely, $h$ satisfies the integral formulation
	\begin{equation}\label{eq:Duhamel_quasilinear}
		h(t)=h_0+\int_0^{t} (\mathcal L_{h(s)})^{-1}\,\mathcal N_\lambda(h(s))\,ds,
		\qquad t\ge 0,
	\end{equation}
	where
	\begin{align*}
		\mathcal N_\lambda(h)
		&=
		-\chi\,\Lambda\tanh(\Lambda)\,h
		-\frac{\lambda}{4}\,\Lambda\tanh(\Lambda)\,\partial_{xxxx} h\\
		&\quad+\sigma\chi\Big(\Lambda\tanh(\Lambda)\big(h\,\Lambda\tanh(\Lambda)\,h\big)
		+\partial_x\big(h\,\partial_x h\big)\Big)\\
		&\quad+\sigma\frac{\lambda}{4}\Big(\Lambda\tanh(\Lambda)\big(h\,\Lambda\tanh(\Lambda)\,\partial_{xxxx} h\big)
		+\partial_x\big(h\,\partial_{xxxxx} h\big)\Big),
	\end{align*}
	and, for a given profile $h$, the operator $\mathcal L_h$ is defined by
	\begin{equation}\label{profile:Lh}
		\mathcal{L}_h U
		=
		\big(1-\Theta\,\Lambda\tanh(\Lambda)\,\partial_{xx}\big)U
		+\sigma\Theta\Big(\Lambda\tanh(\Lambda)\big(h\,\Lambda\tanh(\Lambda)\,\partial_{xx}U\big)
		+\partial_x\big(h\,\partial_{xxx}U\big)\Big).
	\end{equation}
\textcolor{black}{Moreover, the solution decays exponentially to zero, namely
\[
\begin{aligned}
\|h(t)\|_{A^0}+\Theta\tanh(1)\|h(t)\|_{A^3}
&\leq
\exp\left[
-\min\left\{
\frac{\chi\tanh(1)}{2},
\frac{\lambda}{8\Theta}
\right\}t
\right]
\\
&\quad\times
\left(
\|h_0\|_{A^0}+\Theta\tanh(1)\|h_0\|_{A^3}
\right).
\end{aligned}
\]
}
\end{theorem}

\begin{proof}[Proof of Theorem \ref{thm:wp_A3_lambda}]
Throughout we work on $\TT$ and assume $\widehat h_0(0)=0$. Since each term in the right-hand side of
\eqref{eq:model_sigma_LHS_full_multiplier_1D} has zero spatial mean, the mean is preserved by the flow.
Set
\[
\textcolor{black}{\mathcal G_0}:=\Lambda\tanh(\Lambda),\qquad \widehat{\textcolor{black}{\mathcal G_0} f}(k)=|k|\tanh(|k|)\,\widehat f(k),
\]
and for a given profile $h$ define $\mathcal L_h$ by \eqref{profile:Lh}. For convenience we also recall that
$
\mathcal L_0:=1-\Theta\,\textcolor{black}{\mathcal G_0}\partial_{xx}.
$

\medskip
\noindent\underline{Step 1: Approximate problem.}
Let $P_N$ be the Fourier projection
\[
P_N f(x):=\sum_{|k|\le N}\widehat f(k)\,\frac{e^{ikx}}{\sqrt{2\pi}}.
\]
We consider the approximate system for $(h^N,U^N)$:
\begin{equation}\label{eq:approx_system_fullproof}
	\begin{cases}
		\partial_t h^N = U^N,\\[0.2em]
		\mathcal L_{h^N}U^N=\mathcal F_N(h^N),\\[0.2em]
		h^N(0)=P_Nh_0,
	\end{cases}
\end{equation}
where
\begin{align}\label{eq:FN_def_fullproof}
	\mathcal F_N(h^{N})
	&=
	-\chi\,\textcolor{black}{\mathcal G_0} h^{N}
	-\dfrac{\lambda}{4}\,\textcolor{black}{\mathcal G_0}\partial_{xxxx}P_N h^{N} +\sigma\chi\Big(\textcolor{black}{\mathcal G_0}(h^{N}\,\textcolor{black}{\mathcal G_0} h^{N})+\partial_x(h^{N}\,\partial_x h^{N})\Big)\notag\\
	&\quad+\sigma\dfrac{\lambda}{4}\Big(\textcolor{black}{\mathcal G_0}(h^{N}\,\textcolor{black}{\mathcal G_0}\partial_{xxxx}P_N h^{N})+\partial_x\big(h^{N}\,\partial_{xxxxx}P_N h^{N}\big)\Big).
\end{align}
Note that $P_Nh_0$ has zero mean. For each fixed $N$, by the same two-step fixed point argument as in
Theorem~\ref{thm:wp_A1} (first solving $\mathcal L_{h^N}U^N=\mathcal F_N(h^N)$ for $U^N$ with $h^N$ frozen,
and then solving $\partial_t h^N=U^N$), there exists a maximal time $T_N\in(0,\infty]$ such that
\[
h^N\in C([0,T_N);A^3\textcolor{black}{(\TT)})
\]
and \eqref{eq:approx_system_fullproof} holds on $[0,T_N)$.

\medskip
\medskip
\noindent\underline{Step 2: A priori estimates and global existence for the approximate system.}
Fix $N$ and omit the superscript $N$ in this step. Assume that on some interval $[0,T)$ we have the bootstrap bound
\begin{equation}\label{eq:bootstrap_A1_short}
	\sup_{t\in[0,T)}\|h(t)\|_{A^1}\le \mathcal C,
\end{equation}
where $\mathcal C>0$ will be chosen later (independently of $N$).
\medskip

\noindent
By the arguments in the proof of Theorem~\ref{thm:wp_A1} we have:
(i) the invertibility estimate for $\mathcal L_0=1-\Theta G_0\partial_{xx}$ on mean-zero data,
\begin{equation}\label{eq:L0_inv_est_short}
	\|\mathcal L_0^{-1}F\|_{A^0}+\Theta\tanh(1)\,\|\mathcal L_0^{-1}F\|_{A^3}
	\le C\,\|F\|_{A^0},
\end{equation}
and (ii) the symbol-cancellation bound
\begin{equation}\label{eq:I_est_short}
	\|I(h,V)\|_{A^0}\le C\,\|h\|_{A^1}\,\|V\|_{A^3},
	\qquad
	I(h,V):=G_0\big(h\,G_0\partial_{xx}V\big)+\partial_x\big(h\,\partial_{xxx}V\big),
\end{equation}
with constants independent of $N$. From $\mathcal L_h(\partial_t h)=\mathcal F_N(h)$ we write
\[
\mathcal L_0(\partial_t h)=\mathcal F_N(h)-\sigma\Theta\,I(h,\partial_t h).
\]
Applying \eqref{eq:L0_inv_est_short} and using \eqref{eq:I_est_short} together with the bootstrap
\eqref{eq:bootstrap_A1_short} gives
\begin{align}
	\|\partial_t h\|_{A^0}+\Theta\tanh(1)\|\partial_t h\|_{A^3}
	&\le C\Big(\|\mathcal F_N(h)\|_{A^0}+\sigma\Theta\,\|I(h,\partial_t h)\|_{A^0}\Big)\notag\\
	&\le C\Big(\|\mathcal F_N(h)\|_{A^0}+\|h\|_{A^1}\,\|\partial_t h\|_{A^3}\Big)\notag\\
	&\le C\|\mathcal F_N(h)\|_{A^0}+C\mathcal C\,\|\partial_t h\|_{A^3}.
	\label{eq:dth_mixed_preabs}
\end{align}
Choosing $\mathcal C>0$ so that $C\mathcal C\le \frac12\,\Theta\tanh(1)$, we absorb the last term and obtain
\begin{equation}\label{eq:dth_mixed_closed}
	\|\partial_t h\|_{A^0}+\frac{\Theta\tanh(1)}{2}\,\|\partial_t h\|_{A^3}
	\le C\,\|\mathcal F_N(h)\|_{A^0}.
\end{equation}

\smallskip
\noindent
Using the Wiener algebra property, boundedness of $G_0$ on Wiener spaces, and $\|P_N f\|_{A^m}\le \|f\|_{A^m}$,
we have the pointwise estimate
\begin{equation}\label{eq:FN_A0_bound_short}
	\|\mathcal F_N(h)\|_{A^0}
	\le C\Big(\|h\|_{A^1}+\lambda\|P_Nh\|_{A^5}
	+\sigma\|h\|_{A^1}^2+\sigma\lambda\|h\|_{A^1}\|P_Nh\|_{A^5}\Big).
\end{equation}
Proceeding exactly as in Theorem~\ref{thm:wp_A1} (i.e.\ using $\partial_t|\widehat h(k)|\le |\widehat U(k)|$,
the cancellation estimate \eqref{eq:I_est_short} to control the $\sigma\Theta$--contribution, and the lower bound
$\tanh(|k|)\ge \tanh(1)$ for $k\neq 0$), we obtain the differential inequality
\begin{multline}
	\frac{d}{dt}\textcolor{black}{\left(\|h\|_{A^0}+\Theta\tanh(1)\|h\|_{A^3}\right)}
	+\chi\tanh(1)\|h\|_{A^1}
	+\frac{\lambda\tanh(1)}{4}\|P_Nh\|_{A^5}
	\\\le C\Big(\sigma\|h\|_{A^1}^2+\sigma\lambda\|h\|_{A^1}\|P_Nh\|_{A^5}\Big).
	\label{eq:A0_dissip_short}
\end{multline}
Shrinking $\mathcal C$ if necessary so that $C\sigma\mathcal C\le \frac12\chi\tanh(1)$ and
$C\sigma\lambda\mathcal C\le \frac12\cdot \frac{\lambda\tanh(1)}{4}$, the right-hand side of
\eqref{eq:A0_dissip_short} can be absorbed, and we obtain
\begin{equation}\label{eq:A0_absorbed_short}
	\frac{d}{dt}\textcolor{black}{\left(\|h\|_{A^0}+\Theta\tanh(1)\|h\|_{A^3}\right)}
	+\frac{\chi\tanh(1)}{2}\|h\|_{A^1}
	+\frac{\lambda\tanh(1)}{8}\|P_Nh\|_{A^5}
	\le 0.
\end{equation}
\textcolor{black}{
Integrating \eqref{eq:A0_absorbed_short} and using
$\|P_Nh_0\|_{A^s}\leq\|h_0\|_{A^s}$ gives, for every $t<T_N$,
\begin{multline}\label{eq:lambda_integrated_minimal}
\|h(t)\|_{A^0}+\Theta\tanh(1)\|h(t)\|_{A^3}
+\frac{\chi\tanh(1)}2\int_0^t\|h(s)\|_{A^1}\,ds\\
+\frac{\lambda\tanh(1)}8\int_0^t\|h(s)\|_{A^5}\,ds
\leq \|h_0\|_{A^0}+\Theta\tanh(1)\|h_0\|_{A^3}.
\end{multline}
}
\textcolor{black}{From these bounds, we find that}
\begin{equation}\label{estimateA3}
\sup_{t\in[0,T)}\|h(t)\|_{A^3}\le C\big(\|h_0\|_{A^3},\|h_0\|_{A^0},\lambda,\chi,\Theta,\sigma\big),
\end{equation}
with a constant independent of $N$. \medskip

The local well-posedness argument (Step~1) provides a continuation criterion depending only on
$\sup_{[0,t]}\|h\|_{A^1}$. In particular, the solution can be continued as long as
\eqref{eq:bootstrap_A1_short} holds. Since 
\textcolor{black}{
$$
\sup_{t\in[0,\infty)}\|h(t)\|_{A^1}\leq C(\|h_0\|_{A^0}+\Theta\tanh(1)\|h_0\|_{A^3})\leq \mathcal{C}
$$
if $\mathcal{C}_2$ is small enough which guarantees that the maximal time satisfies $T_N=\infty$.
}
\medskip

\noindent\underline{Step 3: Uniform bounds and passage to the limit
$N\to\infty$.}
Fix $T>0$. Integrating \eqref{eq:A0_absorbed_short} and using
$\|P_Nh_0\|_{A^s}\leq\|h_0\|_{A^s}$, we obtain
\textcolor{black}{
\begin{align*}
&\|h^N(t)\|_{A^0}
+\Theta\tanh(1)\|h^N(t)\|_{A^3}
+\frac{\chi\tanh(1)}{2}
 \int_0^t\|h^N(s)\|_{A^1}\,ds\\
&\qquad
+\frac{\lambda\tanh(1)}{8}
 \int_0^t\|h^N(s)\|_{A^5}\,ds
\leq
\|h_0\|_{A^0}
+\Theta\tanh(1)\|h_0\|_{A^3}.
\end{align*}
Consequently,
\[
\sup_N\left(
\|h^N\|_{L^\infty(0,T;A^3)}
+\|h^N\|_{L^1(0,T;A^5)}
\right)<\infty.
\]}
Moreover, combining \eqref{eq:dth_mixed_closed} with
\eqref{eq:FN_A0_bound_short} gives
\[
\sup_N\|\partial_t h^N\|_{L^1(0,T;A^3)}<\infty.
\]

Since $A^5(\TT)$ is compactly embedded into $A^3(\TT)$, a standard
compactness argument yields, after extracting a subsequence,
\[
h^N\longrightarrow h
\qquad\text{strongly in }L^1(0,T;A^3(\TT)).
\]
The preceding uniform bounds and Lemma~\ref{LemmaBilinear} justify the
limit in the nonlinear terms. The limiting equation gives
$\partial_t h\in L^1(0,T;A^3)$, and therefore
\[
h\in W^{1,1}(0,T;A^3)
\subset C([0,T];A^3),
\qquad h(0)=h_0.
\]
Since $T>0$ is arbitrary, the solution is global. \textcolor{black}{Finally, for zero-mean functions,
\[
\frac{\chi\tanh(1)}{2}\|h\|_{A^1}
+\frac{\lambda\tanh(1)}{8}\|h\|_{A^5}
\geq
\min\left\{
\frac{\chi\tanh(1)}{2},
\frac{\lambda}{8\Theta}
\right\}
\left(
\|h\|_{A^0}
+\Theta\tanh(1)\|h\|_{A^3}
\right).
\]
Therefore, \eqref{eq:A0_absorbed_short} and Gr\"onwall's inequality
give
\begin{align*}
\|h(t)\|_{A^0}+\Theta\tanh(1)\|h(t)\|_{A^3}
&\leq
\exp\left[
-\min\left\{
\frac{\chi\tanh(1)}{2},
\frac{\lambda}{8\Theta}
\right\}t
\right]\\
&\quad\times
\left(
\|h_0\|_{A^0}
+\Theta\tanh(1)\|h_0\|_{A^3}
\right),
\end{align*}
which concludes the proof.}

\end{proof}

\begin{remark}\label{rem:infinite_depth}
	Theorem~\ref{thm:wp_A3_lambda} extends to infinite depth by replacing
	\(\textcolor{black}{\mathcal{G}}_0=\Lambda\tanh(\Lambda)\) with \(\textcolor{black}{\mathcal{G}}_\infty=\Lambda\).
	All Wiener-space multiplier bounds and the mean-zero coercivity used in the proof remain valid
	(in fact \(\tanh(1)\) can be replaced by \(1\)), so the invertibility, commutator estimate, and the
	energy argument carry over verbatim. Hence one obtains the same global well-posedness and decay in
	\(A^0\). \textcolor{black}{In this case the differentiated quantity is simply
	$\|h\|_{A^0}+\Theta\|h\|_{A^3}$, and one obtains the same global
	well-posedness and decay in $A^3$.}
\end{remark}

Finally, we turn to the model \eqref{eq:model_sigma_LHS_full_multiplier_1D2}--\eqref{eq:model_sigma_LHS_full_multiplier_1D2b}.
In this case the equation does not contain a nonlinear spatial operator acting on the time derivative,
which slightly simplifies the well--posedness argument. In particular, one has the following global result
for small data.

\begin{theorem}
	\label{thm:wp_A3_model2}
	Fix $\lambda,\chi,\Theta,\sigma>0$ and let $h_0\in A^3(\TT)$ have zero mean.
	There exist $\mathcal C=\mathcal C(\lambda,\chi,\Theta,\sigma),\\
	\textcolor{black}{\mathcal C_2=\mathcal C_2(\lambda,\chi,\Theta,\sigma)}>0$ such that if
	\[
	\|h_0\|_{A^1}\le \mathcal C,
	\]
	and
	\textcolor{black}{
	$$
\|h_0\|_{A^0}+\Theta\tanh(1)\|h_0\|_{A^3}\leq \mathcal{C}_2,
$$
}
	then \textcolor{black}{system} \eqref{eq:model_sigma_LHS_full_multiplier_1D2}-\eqref{eq:model_sigma_LHS_full_multiplier_1D2b}  admits a global solution
	\[
	h\in C\big([0,\infty);A^3(\TT)\big).
	\]
	\textcolor{black}{Moreover, the solution decays exponentially to zero, namely
	\[
	\begin{aligned}
\|h(t)\|_{A^0}+\Theta\tanh(1)\|h(t)\|_{A^3}
&\leq
\exp\left[
-\min\left\{
\frac{\chi\tanh(1)}{2},
\frac{\lambda}{8\Theta}
\right\}t
\right]
\\
&\quad\times
\left(
\|h_0\|_{A^0}+\Theta\tanh(1)\|h_0\|_{A^3}.
\right)
\end{aligned}.
\]}
\end{theorem}

\begin{proof}[Proof of Theorem \ref{thm:wp_A3_model2}]
	The argument follows the same strategy as in Theorems \ref{thm:wp_A3_lambda}. Thus we only indicate the key a priori estimate.
	Using the commutator/symbol-cancellation structure of the nonlinearities and the Wiener algebra bounds,
	one obtains the differential inequality
	\begin{align}\label{eq:energy_model2_preabs}
		\frac{d}{dt}\Big(\|h\|_{A^0}+\Theta\tanh(1)\|h\|_{A^3}\Big)
		+\chi\|h\|_{A^1}+\frac{\lambda\tanh(1)}{2}\|h\|_{A^5}
		\le C\Big(\|h\|_{A^1}^2+\|h\|_{A^1}\|h\|_{A^5}+\|h\|_{A^1}\|\mu\|_{A^3}\Big).
	\end{align}
	From \eqref{eq:model_sigma_LHS_full_multiplier_1D2b} and the multiplier bounds in Wiener spaces we have
	\begin{align}\label{eq:muA3_bound}
		\|\mu\|_{A^3}
		&=\Big\|(1-\Theta\,\Lambda\Delta_x)^{-1}\bigl(-\chi\,\Lambda h-\tfrac{\lambda}{4}\Lambda\partial_{xxxx}h\bigr)\Big\|_{A^3}
		\le C\Big\|-\chi\,\Lambda h-\tfrac{\lambda}{4}\Lambda\partial_{xxxx}h\Big\|_{A^0}
		\le C\|h\|_{A^5},
	\end{align}
	where we also used the mean-zero Poincar\'e inequality to control lower norms by higher ones.
	Inserting \eqref{eq:muA3_bound} into \eqref{eq:energy_model2_preabs} and using the smallness assumption
	$\|h_0\|_{A^1}\le \mathcal C$ (together with a standard continuity/bootstrap argument to propagate
	$\|h(t)\|_{A^1}\le \mathcal C$) we can absorb the right-hand side and obtain
	\begin{equation}\label{eq:energy_model2_absorbed}
		\frac{d}{dt}\Big(\|h\|_{A^0}+\Theta\tanh(1)\|h\|_{A^3}\Big)
		+\frac{\chi}{2}\|h\|_{A^1}+\frac{\lambda\tanh(1)}{4}\|h\|_{A^5}
		\le 0.
	\end{equation}
	This yields uniform-in-time bounds in $A^3$ and integrability of $\|h\|_{A^1}$, which imply global existence \textcolor{black}{due to
$$
\sup_{t\in[0,\infty)}\|h(t)\|_{A^1}\leq C(\|h_0\|_{A^0}+\Theta\tanh(1)\|h_0\|_{A^3})\leq \mathcal{C}
$$
if $\mathcal{C}_2$ is small enough. 
Finally,
\[
\frac{\chi\tanh(1)}{2}\|h\|_{A^1}
+\frac{\lambda\tanh(1)}{8}\|h\|_{A^5}
\geq
\min\left\{
\frac{\chi\tanh(1)}{2},
\frac{\lambda}{8\Theta}
\right\}
\Big(
\|h\|_{A^0}+\Theta\tanh(1)\|h\|_{A^3}
\Big).
\]
Therefore, \eqref{eq:energy_model2_absorbed} and G\"onwall's inequality yields
\[
\begin{aligned}
\|h(t)\|_{A^0}+\Theta\tanh(1)\|h(t)\|_{A^3}
&\leq
\exp\left[
-\min\left\{
\frac{\chi\tanh(1)}{2},
\frac{\lambda}{8\Theta}
\right\}t
\right]
\\
&\quad\times
\left(
\|h_0\|_{A^0}+\Theta\tanh(1)\|h_0\|_{A^3}
\right),
\end{aligned}
\]
give exponential decay.}
\end{proof}

\section{Well-posedness of the weakly nonlinear models for the thin film regime}\label{sec:wp2}
In this section we study the well-posedness of the weakly nonlinear thin-film regime model \eqref{eq:lub_PDE_explicit_1D} introduced in the previous section. For simplicity, we confine the analysis to one space dimension. The result reads as follows

\begin{theorem}\label{thm:wp_lub_A4_statement}
	Fix $\lambda,\chi,\Theta>0$ and parameters $\delta>0$, $\varepsilon\ge 0$.
	Let $h_0\in A^4(\TT)$ be the zero mean initial data.
	There exist constants $\mathcal C=\mathcal C(\lambda,\chi,\Theta,\delta,\varepsilon), \textcolor{black}{\mathcal C_2=\mathcal C_2(\lambda,\chi,\Theta,\delta,\varepsilon)}>0$ such that if
	\[
	\|h_0\|_{A^1}\le \mathcal C,
	\]
	and
\textcolor{black}{$$
	\|h_0\|_{A^0}+\sqrt{\delta}\,\Theta\,\|h_{0}\|_{A^4}\leq \mathcal{C}_2,
	$$
}	
then equation \eqref{eq:lub_PDE_explicit_1D} admits a global (mild) solution
	\[
	h\in C\big([0,\infty);A^4(\TT)\big).
	\]
	Moreover, $h$ satisfies the integral formulation
	\begin{equation}\label{eq:Duhamel_lub}
		h(t)=h_0+\int_0^{t} (\mathcal L_{h(s)})^{-1}\,\mathcal N_{\delta,\varepsilon}\big(h(s)\big)\,ds,
		\qquad t\ge 0,
	\end{equation}
	where
	\[
	\mathcal N_{\delta,\varepsilon}(h)
	:=\sqrt{\delta}\,\partial_x\Bigl((1+\varepsilon h)\,
	\partial_x\Bigl(\chi h+\frac{\lambda}{4}\partial_x^4 h\Bigr)\Bigr),
	\]
	and, for a given profile $h$, the operator $\mathcal L_h$ is defined by
	\begin{equation}\label{eq:profile_Lh_lub}
		\mathcal L_h U
		:=
		U+\sqrt{\delta}\,\Theta\,\partial_x\Bigl((1+\varepsilon h)\,\partial_x^3 U\Bigr).
	\end{equation}
	\textcolor{black}{Finally, the solution decays exponentially in time to zero, namely
	\[
\begin{aligned}
\|h(t)\|_{A^0}+\sqrt\delta\,\Theta\|h(t)\|_{A^4}
&\leq
e^{-\rho t}
\left(
\|h_0\|_{A^0}+\sqrt\delta\,\Theta\|h_0\|_{A^4}
\right),\
\rho&:=
\min\left\{
\frac{\sqrt\delta\,\chi}{2},
\frac{\lambda}{8\Theta}
\right\}.
\end{aligned}
\]}
\end{theorem}

\begin{proof}[Proof of Theorem \ref{thm:wp_lub_A4_statement}]
	We work on $\TT$ and assume $\widehat h_0(0)=0$. Since the right-hand side of
	\eqref{eq:lub_PDE_explicit_1D} is a spatial derivative, the mean is preserved.
	
	\medskip
	\noindent\underline{Step 1: Reformulation and Galerkin approximation.}
	Writing \eqref{eq:lub_PDE_explicit_1D} as a first-order system,
	\begin{equation}\label{eq:lub_system}
		\partial_t h = U,
		\qquad
		\mathcal L_h U = \mathcal N_{\delta,\varepsilon}(h),
	\end{equation}
	where
	\begin{equation}\label{eq:Lh_lub}
		\mathcal L_h U
		:=
		U+\sqrt{\delta}\,\Theta\,\partial_x\Big((1+\varepsilon h)\,\partial_x^3 U\Big),
	\end{equation}
	and
	\begin{equation}\label{eq:N_lub}
		\mathcal N_{\delta,\varepsilon}(h)
		:=
		\sqrt{\delta}\,\partial_x\Big((1+\varepsilon h)\,\partial_x\Big(\chi h+\frac{\lambda}{4}\partial_x^4 h\Big)\Big).
	\end{equation}
	Equivalently,
	\[
	\mathcal L_h U
	=
	U+\sqrt{\delta}\,\Theta\,\partial_x^4U+\sqrt{\delta}\,\Theta\,\varepsilon\,\partial_x\big(h\,\partial_x^3U\big).
	\]
	
	Let $P_N$ be the Fourier projection on $\{|k|\le N\}$. We consider the regularized system
	\begin{equation}\label{eq:lub_system_N}
		\partial_t h^N = U^N,
		\qquad
		\mathcal L_{h^N} U^N = \mathcal N_{\delta,\varepsilon}^N(h^N),
		\qquad
		h^N(0)=P_N h_0,
	\end{equation}
	where the only modification is in the highest derivative,
	\begin{equation}\label{eq:N_lub_N}
		\mathcal N_{\delta,\varepsilon}^N(h)
		:=
		\sqrt{\delta}\,\partial_x\Big((1+\varepsilon h)\,\partial_x\Big(\chi h+\frac{\lambda}{4}\partial_x^4 P_N h\Big)\Big).
	\end{equation}
	
	Fix $T>0$ and $\mathcal C>0$ and define
	\[
	\mathbb X_T^{\mathcal C}
	:=
	\Big\{
	h\in L^\infty(0,T;A^4):\ h(0)=P_Nh_0,\ \sup_{t\in[0,T]}\|h(t)\|_{A^1}\le \textcolor{black}{2}\mathcal C
	\Big\}.
	\]
	As in the previous well-posedness arguments, we construct solutions to \eqref{eq:lub_system_N}
	by two nested fixed point procedures: first solve the elliptic problem for $U^N$ with $h^N$ frozen,
	and then solve $\partial_t h^N = U^N$. \textcolor{black}{For the sake of simplicity, we just 
	write $h, U$ instead of $h^{N}, U^{N}$.}
	
	\medskip
	\noindent\underline{Step 2: Elliptic problem for $U$ with frozen $h$.}
	Fix $h\in\mathbb X_T^{\mathcal C}$. To solve $\mathcal L_h U=\mathcal N_{\delta,\varepsilon}^N(h)$, we use a
	contraction in $A^4$ based on the decomposition
	\[
	\mathcal L_h U
	=
	\underbrace{\big(I+\sqrt{\delta}\,\Theta\,\partial_x^4\big)U}_{=:\mathcal L_\star U}
	+\sqrt{\delta}\,\Theta\,\varepsilon\,\partial_x\big(h\,\partial_x^3U\big).
	\]
	The operator $\mathcal L_\star$ is invertible on mean-zero functions and satisfies the Wiener estimate
	\begin{equation}\label{eq:Lstar_inv}
		\|\mathcal L_\star^{-1}F\|_{A^0}+\sqrt{\delta}\,\Theta\,\|\mathcal L_\star^{-1}F\|_{A^4}
		\le C\,\|F\|_{A^0}.
	\end{equation}
	For $V\in A^4$, define $S_h(V)=U$ by
	\begin{equation}\label{eq:Sh_lub}
		\mathcal L_\star U
		=
		\mathcal N_{\delta,\varepsilon}^N(h)
		-\sqrt{\delta}\,\Theta\,\varepsilon\,\partial_x\big(h\,\partial_x^3V\big).
	\end{equation}
	Using the Wiener algebra property and $\|P_N f\|_{A^m}\le \|f\|_{A^m}$, we have the bounds
	\begin{align}
		\|\partial_x(h\,\partial_x^3V)\|_{A^0}
		&\le C\big(\|h\|_{A^1}\|V\|_{A^3}+\|h\|_{A^0}\|V\|_{A^4}\big), \label{eq:pert_bound}\\
		\|\mathcal N_{\delta,\varepsilon}^N(h)\|_{A^0}
		&\le C\sqrt{\delta}\Big((1+\|h\|_{A^0})\big(\|h\|_{A^2}+\|P_Nh\|_{A^6}\big)+\|h\|_{A^1}\|P_Nh\|_{A^5}\Big).
		\label{eq:Nbound}
	\end{align}
	Applying \eqref{eq:Lstar_inv} to \eqref{eq:Sh_lub} and combining with \eqref{eq:pert_bound}--\eqref{eq:Nbound} yields
	\begin{align}\label{eq:U_est_pre}
		\|U\|_{A^0}+\sqrt{\delta}\,\Theta\,\|U\|_{A^4}
		&\le C\sqrt{\delta}\Big((1+\|h\|_{A^0})\big(\|h\|_{A^2}+\|P_Nh\|_{A^6}\big)+\|h\|_{A^1}\|P_Nh\|_{A^5}\Big)\notag\\
		&\quad + C\sqrt{\delta}\,\Theta\,\varepsilon\big(\|h\|_{A^1}\|V\|_{A^3}+\|h\|_{A^0}\|V\|_{A^4}\big).
	\end{align}
	Similarly, for $V_1,V_2\in A^4$,
	\[
	\|S_h(V_1)-S_h(V_2)\|_{A^4}
	\le C\,\varepsilon\,\|h\|_{A^1}\,\|V_1-V_2\|_{A^4}.
	\]
	Choosing $\mathcal C$ (hence $\|h\|_{A^1}$) small so that $C\varepsilon\mathcal C<1$, $S_h$ is a contraction on $A^4$.
	Therefore $\mathcal L_h U=\mathcal N_{\delta,\varepsilon}^N(h)$ has a unique solution
	$U=U_h\in A^4$, and from \eqref{eq:U_est_pre} (absorbing the $V$-term at the fixed point) we obtain the closed bound
	\begin{align}\label{eq:U_est_closed}
		\|U\|_{A^0}+\frac{\sqrt{\delta}}{2}\,\Theta\,\|U\|_{A^4}
		\le C\sqrt{\delta}\Big((1+\|h\|_{A^0})\big(\|h\|_{A^2}+\|P_Nh\|_{A^6}\big)+\|h\|_{A^1}\|P_Nh\|_{A^5}\Big).
	\end{align}
	
	\medskip
	\noindent\underline{Step 3: Fixed point for the evolution.}
	Define
	\begin{equation}\label{eq:Th_lub}
		(\mathcal T h)(t):=P_Nh_0+\int_0^t U_h(s)\,ds,\qquad t\in[0,T].
	\end{equation}
	Using \eqref{eq:U_est_closed} and standard Wiener product estimates, one shows (as in Section~\ref{sec:wp2})
	that for $T>0$ small enough (depending on $\mathcal C$ but independent of $N$) the map
	$\mathcal T$ sends $\mathbb X_T^{\mathcal C}$ into itself and is a contraction in
	$L^\infty(0,T;A^1)$. Hence there exists a unique local solution $(h^N,U^N)$ to \eqref{eq:lub_system_N}
	on a maximal interval $[0,T_N)$, with $h^N\in C([0,T_N);A^4)$.
	
	\medskip
	\noindent\underline{Step 4: A priori estimates, global existence, and decay.}
	Define the energy
	\[
	E(t):=\|h(t)\|_{A^0}+\sqrt{\delta}\,\Theta\,\|h(t)\|_{A^4}.
	\]
	Using \textcolor{black}{that the approximate problem can be written as
$$
		\partial_t h+\sqrt{\delta}\,\partial_x^{2}\left(\Theta\,\partial_{xx}\partial_t h-\chi h-\frac{\lambda}{4}\,\partial_x^{4}P_Nh\right)
		=-\sqrt{\delta}\,\varepsilon\,\partial_x\big(h\,\partial_x\left(\Theta\,\partial_{xx}\partial_t h-\chi h-\frac{\lambda}{4}\,\partial_x^{4}P_Nh\right)\big),
$$
}	
one obtains the differential inequality
\begin{align}\label{eq:energy_lub_diff_full}
	\frac{d}{dt}E(t)
	+\sqrt{\delta}\,\chi\,\|h(t)\|_{A^2}
	+\sqrt{\delta}\,\frac{\lambda}{4}\,\|P_N h(t)\|_{A^6}
	&\le
	C\,\|h(t)\|_{A^1}
	\Big(
	\|h(t)\|_{A^2}
	+\|P_N h(t)\|_{A^6}
	+\|h(t)\|_{A^1} \notag\\
	&\qquad\qquad
	+\|P_N h(t)\|_{A^5}
	+\|h(t)\|_{A^1}\,\|P_N h(t)\|_{A^6}
	\Big).
\end{align}
\textcolor{black}{Indeed, in the previous step we have used the product estimate of the Wiener spaces together with Poincar\'e inequality, the smallness of the solution in $A^1$ and the elliptic estimate \eqref{eq:U_est_closed} to find that
\begin{align*}
I&=\|\sqrt{\delta}\,\varepsilon\,\partial_x\big(h\,\partial_x\left(\Theta\,\partial_{xx}\partial_t h-\chi h-\frac{\lambda}{4}\,\partial_x^{4}P_Nh\right)\big)\|_{A^0}\\
&\leq C\|h\|_{A^1}\left(\|\partial_t h\|_{A^4}+\|h\|_{A^2}+\|P_N h\|_{A^6}\right)\\
&\leq C\|h\|_{A^1}\left(\Big((1+\|h\|_{A^0})\big(\|h\|_{A^2}+\|P_Nh\|_{A^6}\big)+\|h\|_{A^1}\|P_Nh\|_{A^5}\Big)+\|h\|_{A^2}+\|P_N h\|_{A^6}\right).
\end{align*}
As before, choosing $\mathcal C$ small enough allows us to absorb the right-hand side into the left-hand side of
	\eqref{eq:energy_lub_diff_full}, which} yields
	\begin{equation}\label{eq:energy_lub_abs}
		\frac{d}{dt}E(t)
		+\sqrt{\delta}\,\frac{\chi}{2}\,\|h(t)\|_{A^2}
		+\sqrt{\delta}\,\frac{\lambda}{8}\,\|P_N h(t)\|_{A^6}
		\le 0.
	\end{equation}
	In particular, $E(t)\le E(0)$ for all $t<T_N$, giving uniform-in-$N$ control of $h^N$ in $A^4$ and decay of
	$\|h(t)\|_{A^4}$. 
	\textcolor{black}{By interpolation, we have that
	\[
	\min\{1,\sqrt\delta\,\Theta\}\|h\|_{A^1}
	\leq\|h\|_{A^0}+\sqrt\delta\,\Theta\|h\|_{A^4}.
	\]
Moreover, due to the fact that
$$
\sup_{t\in[0,\infty)}\|h(t)\|_{A^1}\leq C(\|h_0\|_{A^0}+\sqrt\delta\,\Theta\|h_0\|_{A^4})\leq 2\mathcal{C},
$$
if $\mathcal{C}_2$ is small enough,} a standard continuation argument then shows that the bootstrap persists globally,
	hence $T_N=\infty$.
	
	\medskip
	
	\medskip
\noindent\underline{Step 5: Passage to the limit $N\to\infty$.}
Fix $T>0$. From \eqref{eq:energy_lub_abs} and the preceding estimate
for $\partial_t h^N$, we have
\[
\sup_N\left(
\|h^N\|_{L^\infty(0,T;A^4)}
+\|h^N\|_{L^1(0,T;A^6)}
+\|\partial_t h^N\|_{L^1(0,T;A^4)}
\right)<\infty.
\]

The compactness argument used in the preceding theorem applies here
with $A^6(\TT)$ compactly embedded into $A^4(\TT)$. Thus, after
extracting a subsequence,
\[
h^N\longrightarrow h
\qquad\text{strongly in }L^1(0,T;A^4(\TT)).
\]
The uniform bounds and the Wiener product estimates allow us to pass to
the limit as before. The limiting equation
then gives $\partial_t h\in L^1(0,T;A^4)$, and hence
\[
h\in W^{1,1}(0,T;A^4)
\subset C([0,T];A^4),
\qquad h(0)=h_0.
\]
Since $T>0$ is arbitrary, the solution is global.

\textcolor{black}{
Finally, for zero-mean functions,
\[
\frac{\sqrt{\delta}\,\chi}{2}\|h\|_{A^2}
+\frac{\sqrt{\delta}\,\lambda}{8}\|h\|_{A^6}
\geq
\rho\left(
\|h\|_{A^0}
+\sqrt{\delta}\,\Theta\|h\|_{A^4}
\right),
\qquad
\rho:=
\min\left\{
\frac{\sqrt{\delta}\,\chi}{2},
\frac{\lambda}{8\Theta}
\right\}.
\]
Therefore, \eqref{eq:energy_lub_abs} and Gr\"onwall's inequality give
\[
\|h(t)\|_{A^0}
+\sqrt{\delta}\,\Theta\|h(t)\|_{A^4}
\leq
e^{-\rho t}
\left(
\|h_0\|_{A^0}
+\sqrt{\delta}\,\Theta\|h_0\|_{A^4}
\right),
\]
which proves the desired exponential decay.}
\end{proof}

\section*{Acknowledgement} 
D.A.-O.\ has been partially supported by grant RYC2023-045563-I (MICIU/AEI/10.13039/501100011033 and ESF+) and by project PID2023-148028NB-I00. D.A.-O.\ and R.G.-B.\ are supported by the projects ``An\'alisis Matem\'atico Aplicado
y Ecuaciones Diferenciales" Grant PID2022-141187NB-I00 and ``EDPs Deterministas y Estoc\'asticas en Biolog\'ia y Mec\'anica de Fluidos" Grant PID2025-171734NA-I00 funded by MCIN/AEI/10.13039/501100011033/FEDER, UE.
\textcolor{black}{Both authors are grateful to the referees for their thorough evaluation and insightful feedback, which greatly enhanced the quality of our paper.}


\begin{thebibliography}{20}
	
	\bibitem{APW2023}
	S.~Agrawal, N.~Patel, and S.~Wu,
	\newblock \emph{Rigidity of acute angled corners for one phase Muskat interfaces},
	\newblock Adv. Math. \textbf{412} (2023), 108801.
	
	\bibitem{AlazardOneFluid2019}
	T.~Alazard and Q.-H.~Nguyen,
	\newblock \emph{The Cauchy problem for the Muskat equation: a critical initial data},
	\newblock Comm. Math. Phys. \textbf{370} (2019), 1043--1102.
	
	\bibitem{AlazardHung2023}
	T.~Alazard and Q.-H.~Nguyen,
	\newblock \emph{Muskat equation: identities and the Cauchy problem},
	\newblock Comm. Math. Phys. \textbf{377} (2020), 1421--1459.
	
	\bibitem{Ambrose2014}
	D.~M.~Ambrose,
	\newblock \emph{The zero surface tension limit of two-dimensional interfacial Darcy flow},
	\newblock J. Math. Fluid Mech. \textbf{16} (2014), 905--143.
	
	\bibitem{MR3656704}
	D.~M.~Ambrose and M.~Siegel,
	\newblock \emph{Well-posedness of two-dimensional hydroelastic waves},
	\newblock Proc. Roy. Soc. Edinburgh Sect. A \textbf{147} (2017), no.~3, 529--570.
	
	\bibitem{Bear1988}
	J.~Bear,
	\newblock \emph{Dynamics of Fluids in Porous Media},
	\newblock Dover Publications, 1988.
	
	\bibitem{BBD-P1995}
	F.~Bernis, L.~A.~Beretta, M.~Bertsch, and R.~Dal Passo,
	\newblock \emph{Nonnegative solutions of a fourth-order nonlinear degenerate parabolic equation},
	\newblock Arch. Ration. Mech. Anal. \textbf{129} (1995), 175--200.
	
	\bibitem{BF1990}
	J.~Bernis and A.~Friedman,
	\newblock \emph{Higher order nonlinear degenerate parabolic equations},
	\newblock J. Differential Equations \textbf{83} (1990), 179--206.
	
	\bibitem{BP1996}
	A.~L.~Bertozzi and M.~Pugh,
	\newblock \emph{The lubrication approximation for thin viscous films: regularity and long-time behavior of weak solutions},
	\newblock Comm. Pure Appl. Math. \textbf{49} (1996), 85--123.
	
	\bibitem{BG-B2019}
	G.~Bruell and R.~Granero-Belinch\'on,
	\newblock \emph{On the thin film Muskat and the thin film Stokes equations},
	\newblock J. Math. Fluid Mech. \textbf{21} (2019), 21:33.
	
	\bibitem{BocchiGancedo2022}
	E.~Bocchi and F.~Gancedo,
	\newblock \emph{Rigorous thin film approximations of the one-phase unstable Muskat problem},
	\newblock Indiana Univ. Math. J. \textbf{73} (2024), no.~5, 1747--1796.

	\bibitem{BukalMuha2021}
	M. Bukal and B. Muha,
	\newblock \emph{Rigorous derivation of a linear sixth-order thin film equation as a reduced model for thin fluid -- thin structure interaction problems,}
	\newblock Applied Mathematics \& Optimization \textbf{84}(2) (2021), 2245--2288.
	
\bibitem{BukalMuha2022}
	M. Bukal and B. Muha,
	\newblock \emph{Justification of a nonlinear sixth-order thin-film equation as the reduced model for a fluid -- structure interaction problem},
	\newblock Nonlinearity \textbf{35}(8) (2022), 4695--4726.
	
	\bibitem{MR4673875}
	S.~Cameron and R.~M.~Strain,
	\newblock \emph{Critical local well-posedness for the fully nonlinear Peskin problem},
	\newblock Comm. Pure Appl. Math. \textbf{77} (2024), no.~2, 901--989.
	
	\bibitem{CCFGL-F2012}
	A.~Castro, D.~C\'ordoba, C.~Fefferman, F.~Gancedo, and M.~L\'opez-Fern\'andez,
	\newblock \emph{Rayleigh--Taylor breakdown for the Muskat problem with applications to water waves},
	\newblock Ann. of Math. \textbf{175} (2012), 909--948.
	
	\bibitem{CCFG2016}
	A.~Castro, D.~C\'ordoba, C.~Fefferman, and F.~Gancedo,
	\newblock \emph{Splash singularities for the one-phase Muskat problem in stable regimes},
	\newblock Arch. Ration. Mech. Anal. \textbf{222} (2016), 213--243.
	
	\bibitem{Chen1993}
	Y.~Chen,
	\newblock \emph{The Hele--Shaw problem and area-preserving curve-shortening motions},
	\newblock Arch. Ration. Mech. Anal. \textbf{123} (1993), 117--151.
	
	\bibitem{Con1993}
	P.~Constantin, T.~F.~Dupont, R.~E.~Goldstein, L.~P.~Kadanoff, M.~J.~Shelley, and S.-M.~Zhou,
	\newblock \emph{Droplet breakup in a model of the Hele--Shaw cell},
	\newblock Phys. Rev. E \textbf{47} (1993), 4169--4192.
	
	\bibitem{ConstantinPugh1993}
	P.~Constantin and M.~Pugh,
	\newblock \emph{Global solutions for small data to the Hele--Shaw problem},
	\newblock Nonlinearity \textbf{6} (1993), 393--415.
	
	\bibitem{Con2018}
	P.~Constantin, T.~Elgindi, H.~Nguyen, and V.~Vicol,
	\newblock \emph{On singularity formation in a Hele--Shaw model},
	\newblock Comm. Math. Phys. \textbf{363} (2018), 139--171.
	
	\bibitem{CordobaPernas-Castano2017}
	D.~C\'ordoba and T.~Pernas-Casta\~no,
	\newblock \emph{Non-splat singularity for the one-phase Muskat problem},
	\newblock Trans. Amer. Math. Soc. \textbf{369} (2017), 711--754.
	
	\bibitem{DGN2021}
	H.~Dong, F.~Gancedo, and H.~Q.~Nguyen,
	\newblock \emph{Global well-posedness for the one-phase Muskat problem},
	\newblock Comm. Pure Appl. Math. \textbf{76} (2023), no.~12, 3912--3967.
	
	\bibitem{DuchonRobert1984}
	J.~Duchon and R.~Robert,
	\newblock \emph{\'Evolution d'une interface par capillarit\'e et diffusion de volume. I. Existence locale en temps},
	\newblock Ann. Inst. H. Poincar\'e Anal. Non Lin\'eaire \textbf{1} (1984), 361--378.
	
	\bibitem{ELM2011}
	J.~Escher, P.~Laurencot, and B.-V.~Matioc,
	\newblock \emph{Existence and stability of weak solutions for a degenerate parabolic system modelling two-phase flows in porous media},
	\newblock Ann. Inst. H. Poincar\'e Anal. Non Lin\'eaire \textbf{28} (2011), 583--598.
	
	\bibitem{EMM2012}
	J.~Escher, A.-V.~Matioc, and B.-V.~Matioc,
	\newblock \emph{Modelling and analysis of the Muskat problem for thin fluid layers},
	\newblock J. Math. Fluid Mech. \textbf{14} (2012), 267--277.
	
	\bibitem{EM2013}
	J.~Escher and B.-V.~Matioc,
	\newblock \emph{Existence and stability of solutions for a strongly coupled system modelling thin fluid films},
	\newblock NoDEA Nonlinear Differential Equations Appl. \textbf{20} (2013), 539--555.
	
	\bibitem{EscherMatioc2011}
	J.~Escher and B.-V.~Matioc,
	\newblock \emph{On the parabolicity of the Muskat problem: well-posedness, fingering, and stability results},
	\newblock Z. Anal. Anwend. \textbf{30} (2011), 193--218.
	
	\bibitem{EscherSimonett1997}
	J.~Escher and G.~Simonett,
	\newblock \emph{Classical solutions for Hele--Shaw models with surface tension},
	\newblock Adv. Differential Equations \textbf{2} (1997), 619--642.
	
	\bibitem{FlynnNguyen2021}
	D.~Flynn and H.~Q.~Nguyen,
	\newblock \emph{The vanishing surface tension limit of the Muskat problem},
	\newblock Comm. Math. Phys. \textbf{382} (2021), 1205--1241.
	
	\bibitem{MR4959949}
	M.~Gahn,
	\newblock \emph{Derivation of a Biot plate system for a thin poroelastic layer},
	\newblock SIAM J. Math. Anal. \textbf{57} (2025), no.~5, 5303--5341.
	
	\bibitem{Sema2017}
	F.~Gancedo,
	\newblock \emph{A survey for the Muskat problem and a new estimate},
	\newblock SeMA J. \textbf{74} (2017), 21--35.
	
	\bibitem{GG-JPS2019}
	F.~Gancedo, E.~Garc\'ia-Ju\'arez, N.~Patel, and R.~M.~Strain,
	\newblock \emph{On the Muskat problem with viscosity jump: global in time results},
	\newblock Adv. Math. \textbf{345} (2019), 552--597.
	
	\bibitem{GG-JPS2019Bubble}
	F.~Gancedo, E.~Garc\'ia-Ju\'arez, N.~Patel, and R.~M.~Strain,
	\newblock \emph{Global regularity for gravity unstable Muskat bubbles},
	\newblock Memoirs of the American Mathematical Society, \textbf{292} (2023), no.~1455, v+87 pp.
	
	\bibitem{GancedoLazar2022}
	F.~Gancedo and O.~Lazar,
	\newblock \emph{Global well-posedness for the three dimensional Muskat problem in the critical Sobolev space},
	\newblock Arch. Ration. Mech. Anal. \textbf{246} (2022), 1141--207.
	
	\bibitem{GG-BS2020}
	F.~Gancedo, R.~Granero-Belinch\'on, and M.~Scrobogna,
	\newblock \emph{Surface tension stabilization of the Rayleigh--Taylor instability for a fluid layer in a porous medium},
	\newblock Ann. Inst. H. Poincar\'e Anal. Non Lin\'eaire \textbf{37} (2020), 1299--1343.
	
	\bibitem{GancedoStrain2014}
	F.~Gancedo and R.~M.~Strain,
	\newblock \emph{Absence of splash singularities for surface quasi-geostrophic sharp fronts and the Muskat problem},
	\newblock Proc. Natl. Acad. Sci. USA \textbf{111} (2014), 635--639.
	
	\bibitem{GraneroBelinchonScrobogna2019}
	R.~Granero-Belinch\'on and S.~Scrobogna,
	\newblock \emph{Asymptotic models for free boundary flow in porous media},
	\newblock Physica D \textbf{392} (2019), 1--16.
	
	\bibitem{GraneroBelinchonScrobogna2020}
	R.~Granero-Belinch\'on and S.~Scrobogna,
	\newblock \emph{On an asymptotic model for free boundary Darcy flow in porous media},
	\newblock SIAM J. Math. Anal. \textbf{52} (2020), no.~5, 4937--4970.
	
	\bibitem{GHS2007}
	Y.~Guo, C.~Hallstrom, and D.~Spirn,
	\newblock \emph{Dynamics near unstable, interfacial fluids},
	\newblock Comm. Math. Phys. \textbf{270} (2007), 635--689.
	
	\bibitem{HQNguyen2019}
	H.~Q.~Nguyen,
	\newblock \emph{On well-posedness of the Muskat problem with surface tension},
	\newblock Adv. Math. \textbf{374} (2020), 107344.
	
	\bibitem{Kim2003}
	I.~C.~Kim,
	\newblock \emph{Uniqueness and existence results on the Hele--Shaw and the Stefan problems},
	\newblock Arch. Ration. Mech. Anal. \textbf{168} (2003), 299--328.
	
	\bibitem{MR2812939}
	A.~Korobkin, E.~I.~P\u{a}r\u{a}u, and J.-M.~Vanden-Broeck,
	\newblock \emph{The mathematical challenges and modelling of hydroelasticity},
	\newblock Philos. Trans. R. Soc. Lond. Ser. A Math. Phys. Eng. Sci. \textbf{369} (2011), no.~1947, 2803--2812.
	
	\bibitem{Lannes2013}
	D.~Lannes,
	\newblock \emph{The Water Waves Problem: Mathematical Analysis and Asymptotics},
	\newblock Mathematical Surveys and Monographs, Vol.~188,
	\newblock American Mathematical Society, Providence, RI, 2013.
	
	\bibitem{LM2013}
	P.~Laurencot and B.-V.~Matioc,
	\newblock \emph{A gradient flow approach to a thin film approximation of the Muskat problem},
	\newblock Calc. Var. Partial Differential Equations \textbf{47} (2013), 319--341.
	
	\bibitem{LM2014}
	P.~Laurencot and B.-V.~Matioc,
	\newblock \emph{A thin film approximation of the Muskat problem with gravity and capillary forces},
	\newblock J. Math. Soc. Japan \textbf{66} (2014), 1043--1071.
	
	\bibitem{MR3608168}
	S.~Liu and D.~M.~Ambrose,
	\newblock \emph{Well-posedness of two-dimensional hydroelastic waves with mass},
	\newblock J. Differential Equations \textbf{262} (2017), no.~9, 4656--4699.
	
	\bibitem{Mat2012}
	B.-V.~Matioc,
	\newblock \emph{Non-negative global weak solutions for a thin film approximation of the Muskat problem},
	\newblock J. Differential Equations \textbf{252} (2012), 1043--1071.
	
	\bibitem{MP2012}
	B.-V.~Matioc and G.~Prokert,
	\newblock \emph{Hele--Shaw flow in thin threads: a rigorous limit result},
	\newblock Interfaces Free Bound. \textbf{14} (2012), 205--230.
	
	\bibitem{MR2863467}
	A.~Meirmanov,
	\newblock \emph{The Muskat problem for a viscoelastic filtration},
	\newblock Interfaces Free Bound. \textbf{13} (2011), no.~4, 463--484.
	
	\bibitem{Otto1997}
	F.~Otto,
	\newblock \emph{Viscous fingering: an optimal bound on the growth rate of the mixing zone},
	\newblock SIAM J. Appl. Math. \textbf{57} (1997), 982--990.
	
	\bibitem{MR2812947}
	P.~I.~Plotnikov and J.~F.~Toland,
	\newblock \emph{Modelling nonlinear hydroelastic waves},
	\newblock Philos. Trans. R. Soc. Lond. Ser. A Math. Phys. Eng. Sci. \textbf{369} (2011), no.~1947, 2942--2956.
	
	\bibitem{MR2413099}
	J.~F.~Toland,
	\newblock \emph{Steady periodic hydroelastic waves},
	\newblock Arch. Ration. Mech. Anal. \textbf{189} (2008), no.~2, 325--362.
	
	\bibitem{WanYang2026}
	L.~Wan and J.~Yang,
	\newblock \emph{On the well-posedness of two-dimensional Muskat problem with an elastic interface},
	\newblock arXiv:2601.01374 (2026).
	
	\bibitem{MR4104949}
	Z.~Wang and J.~Yang,
	\newblock \emph{Energy estimates and local well-posedness of 3D interfacial hydroelastic waves between two incompressible fluids},
	\newblock J. Differential Equations \textbf{269} (2020), no.~7, 6055--6087.
	
	\bibitem{YeTanveer2011}
	J.~Ye and S.~Tanveer,
	\newblock \emph{Global existence for a translating near-circular Hele--Shaw bubble},
	\newblock SIAM J. Math. Anal. \textbf{43} (2011), 457--506.
	
\end{thebibliography}
 \end{document}